 \documentclass[11pt]{amsart}

\usepackage{amsfonts, amsmath,amscd, amssymb, latexsym, mathrsfs, slashed, stmaryrd, verbatim, wasysym }
\usepackage[all]{xy}
\usepackage{slashed}

\newtheorem*{acknowledgement}{Acknowledgement}
\newtheorem{theorem}{Theorem}
\newtheorem{corollary}{Corollary}

\newtheorem{lemma}[corollary]{Lemma}
\newtheorem{proposition}[corollary]{Proposition}
\newtheorem{remark}[corollary]{Remark}

\newcommand{\area}{\operatorname{area}}
\newcommand{\sus}{\operatorname{sus}}
\newcommand{\sign}{\operatorname{sign}}

\renewcommand{\bar}{\overline}
\renewcommand{\Re}{\operatorname{Re}}
\renewcommand{\hat}[1]{\widehat{#1}}

\newcommand{\rest}[1]{\downharpoonright_{#1}}

\newcommand\pa{\partial}
\newcommand\bpa{\overline{\partial}}
\newcommand{\db}{\bar{\partial}}

\newcommand{\SL}{\operatorname{SL}}
\newcommand{\ta}{\widetilde{a}}

\newcommand\fpint{\overline{\int_{0}^{\infty}}}
\newcommand\cf{cf\@. }

\newcommand\eps\varepsilon
\newcommand\p\phi
\renewcommand\P\Phi
\renewcommand\det{\operatorname{det}}

\newcommand\fD{\operatorname{\phi-hc}}
\newcommand\hc{\operatorname{hc}}

\newcommand\CI{{\mathcal{C}}^{\infty}}

\newcommand{\lrpar}[1]{\left( #1 \right)}

\newcommand\Ch{\operatorname{Ch}}

\newcommand\coker{\operatorname{coker}}

\newcommand\Diff{\operatorname{Diff}}

\DeclareMathOperator*{\FP}{\operatorname{FP}}

\newcommand\Id{\operatorname{Id}}
\newcommand\ind{\operatorname{ind}}
\newcommand\Ind{\operatorname{Ind}}

\newcommand\Str{\operatorname{Str}}
\newcommand\Td{\operatorname{Td}}

\newcommand\Tr{\operatorname{Tr}}
\newcommand\STr{\operatorname{STr}}

\newcommand\cusp{\operatorname{cu}}

\newcommand\bbA{\mathbb{A}}

\newcommand\bbC{\mathbb{C}}

\newcommand\bbH{\mathbb{H}}

\newcommand\bbN{\mathbb{N}}

\newcommand\bbR{\mathbb{R}}
\newcommand\bbS{\mathbb{S}}

\newcommand\bbZ{\mathbb{Z}}

\newcommand\cA{\mathcal{A}}
\newcommand\cB{\mathcal{B}}
\newcommand\cC{\mathcal{C}}

\newcommand\cE{\mathcal{E}}
\newcommand\cF{\mathcal{F}}

\newcommand\cH{\mathcal{H}}

\newcommand\cL{\mathcal{L}}
\newcommand\cM{\mathcal{M}}

\newcommand\cP{\mathcal{P}}

\newcommand\cT{\mathcal{T}}
\newcommand\cU{\mathcal{U}}
\newcommand\cV{\mathcal{V}}

\DeclareMathAlphabet{\mathpzc}{OT1}{pzc}{m}{it}

\newcommand\paperintro%
        {%
         }
\newcommand\paperbody%
        {%
         }

\newcommand{\pt}{\operatorname{pt}}
\newcommand{\even}{\operatorname{ev}}
\newcommand{\Mod}{\operatorname{Mod}}
\newcommand{\PSL}{\operatorname{PSL}}
\newcommand{\pr}{\operatorname{pr}}
\newcommand{\TZ}{\operatorname{TZ}}
\renewcommand{\Im}{\operatorname{Im}}
\renewcommand{\Re}{\operatorname{Re}}

\title[Local families index for $\db$-operators]{A local families index formula for $\db$-operators on punctured
Riemann surfaces}
\author{Pierre Albin}
\author{Fr\'ed\'eric Rochon}
\address{Department of Mathematics, Massachusetts Institute of Technology}
\email{pierre@math.mit.edu} 
\address{Department of Mathematics, University of Toronto}
\email{rochon@math.toronto.edu} 
\thanks{The first author was partially supported by a NSF postdoctoral fellowship. The second author was supported by a postdoctoral fellowship of the Fonds qu\'eb\'ecois de la recherche sur la nature et les technologies.}

\begin{document}


\begin{abstract}
Using heat kernel methods developed by Vaillant, a local index formula is obtained for
families of $\db$-operators on the Teichm\"uller universal curve of Riemann surfaces of genus $g$ with $n$ punctures.  
The formula also holds on the moduli space $\cM_{g,n}$ in the
sense of orbifolds where it can be written in terms of Mumford-Morita-Miller classes.  The degree two part of the formula gives the curvature of the corresponding 
determinant line bundle equipped with the Quillen connection, a result originally obtained by
Takhtajan and Zograf.
\end{abstract}

\maketitle
\tableofcontents

\section*{Introduction}\label{int.0}

Let $X$ be a smooth even dimensional oriented compact manifold with boundary $\pa X\ne \emptyset$.  Assume that the boundary is the total space of a fibration
\begin{equation}
\xymatrix{Z\ar@{-}[r]&\pa X\ar[d]^{\phi}\\
& Y}
\label{int.1}\end{equation}
where $Y$ and $Z$ are compact oriented manifolds, $Y$ being the base and $Z$ being a typical fibre.  Let $x\in \CI(X)$ be a boundary defining function for $X$ and 
\begin{equation}
  c: \pa X\times [0,\epsilon)_{x}\to N\subset X
\label{int.2}\end{equation}
a corresponding collar neighborhood of $\pa X$ in $X$.  Let $g_{\hc}$ be a metric on $X\setminus \pa X$ which takes the form
\begin{equation}
     c^{*} g_{\hc}= \frac{dx}{x^{2}}+ \phi^{*} g_{Y}+ x^{2} g_{Z}
\label{int.3}\end{equation}
in the collar neighborhood \eqref{int.2},
where $g_{Z}$ is a metric for the vertical tangent bundle of \eqref{int.1}
and $g_{Y}$ is a metric on the base $Y$ which is lifted to $\pa X$ using a choice
of connection for the fibration \eqref{int.1}.  Such a metric is called a product fibred hyperbolic cusp metric (product $d$-metric in the terminology of \cite{Vaillant}).  If the
manifolds $X$, $Y$ and $Z$ are spin, one can construct a Dirac operator associated to the
metric $g_{\hc}$.  More generally, one can consider a Dirac type operator $D$ constructed using
the metric $g_{\hc}$ and a Clifford module $E\to X$ with a choice of Clifford connection.  

In his thesis \cite{Vaillant}, Vaillant studied the index and the spectral theory of such
operators.  To do so, he introduced the conformally related operator $xD$ and defined the vertical family by $D^{V}:= \left. xD\right|_{\pa X}$, which is a family of operators on
$\pa X$ parametrized by the base $Y$ and acting on each fibre of \eqref{int.1}.  Assuming 
that the rank of $\ker D^{V}\to Y$ is constant so that it is a vector bundle over $Y$
(constant rank assumption), Vaillant also introduced a horizontal operator
\begin{equation}
   D^{H}: \CI(Y; \ker D^{V})\to \CI(Y;\ker D^{V})
\label{int.4}\end{equation}
which governs the continuous spectrum of $D$ with bands of continuous spectrum starting
at the eigenvalues of $D^{H}$ and going out at infinity.  In particular, the operator $D$ is 
Fredholm if and only if $D^{H}$ is invertible.  In that case, Vaillant was able to obtain
a formula for its index using heat kernel techniques and Getzler's rescaling along the lines
of \cite{APSbook},
\begin{equation}
   \ind(D)= \int_{X} \hat{A}(R_{\hc})\Ch( F^{E/S_{\hc}}) -
     \int_{Y} \hat{A}(R_{g_{Y}})\hat{\eta}(D^{V}) -\frac{1}{2}\eta(D^{H}),
\label{int.5}\end{equation}
the first term being the usual Atiyah-Singer integral, $\hat{\eta}(D^{V})$ being the 
Bismut-Cheeger eta form of the vertical family and $\eta(D^{H})$ being the eta invariant
of $D^{H}$.  

In \cite{Albin-Rochon}, the authors, inspired by the work of Melrose and Piazza in \cite{Melrose-Piazza1} and \cite{Melrose-Piazza2}, generalized the formula of Vaillant
to families of Dirac type operators.  Via the use of Fredholm perturbations, a notion intimately related to spectral sections, it was also possible to study situations where the constant rank
assumption is not satisfied, allowing among other things to generalize the index theorem
of Leichtnam, Mazzeo and Piazza \cite{LMP}.

The present paper, which is a sequel to \cite{Albin-Rochon}, intends to put into use the index
formula of \cite{Albin-Rochon} to study the following fundamental example arising in Teichm\"uller theory.  Assume that $2g+n\ge 3$ and let $T_{g,n}$ be the Teichm\"uller space of 
Riemann surfaces of genus $g$ with $n$ punctures.  Let $p: \cT_{g,n}\to T_{g,n}$ be the Teichm\"uller universal curve whose fibre above $[\Sigma]\in T_{g,n}$ is the
corresponding Riemann surface $\Sigma$ of genus $g$ with $n$ punctures.  Let 
$T^{i,j}_{v}\cT_{g,n}\to \cT_{g,n}$ be the $(i,j)$ vertical tangent bundle and let
$\Lambda^{i,j}_{v}$ be its dual.  In particular, $K_{v}:=\Lambda^{1,0}_{v}$ restricts on each
fibre $\Sigma$ to the corresponding canonical line bundle $K_{\Sigma}:=\Lambda^{1,0}_{\Sigma}$.

For each $\ell\in \bbZ$, one can associate a family of $\db$-operators
\begin{equation}
   \db_{\ell}: \CI(\cT_{g,n}; K_{v}^{\ell})\to \CI( \cT_{g,n}; \Lambda^{0,1}_{v}\otimes K_{v}^{\ell})
\label{int.6}\end{equation}
acting fibre by fibre on $p: \cT_{g,n}\to T_{g,n}$ and parametrized by the base $T_{g,n}$.  By
the uniformization theorem for Riemann surfaces, each fibre $\Sigma$ of $p: \cT_{g,n}\to
T_{g,n}$ comes equipped with a hyperbolic metric $g_{\Sigma}$.  Compactifying each fibre 
by a compact Riemann surface with boundary, these metrics can be seen as product hyperbolic 
cusp metrics, the fibration structure on the boundary being the collapsing map onto a point.
With these metrics, the family $\db_{\ell}$ can be interpreted as a family of Dirac-type operators associated to a family of product hyperbolic cusp metrics.  Using the criterion of
Vaillant \cite{Vaillant}, one can check that each member of the family is Fredholm.  
The formula of \cite{Albin-Rochon} therefore applies.

As described in \cite{Wolpert}, the fibration $p: \cT_{g,n}\to T_{g,n}$ is equipped with
a canonical connection.  This allows one to interpret the formula of \cite{Albin-Rochon} at the
level of forms.  In general, the eta forms involved in this formula are quite hard to compute.  However, in this specific example, an explicit computation is possible using a result of 
Zhang \cite{Zhang}, the vertical family being defined on a circle fibration.  The \textbf{main result of 
this paper}, theorem~\ref{lf.8}, gives the following local family index formula, 
\begin{multline}
\Ch( \Ind(\db_{\ell}))= 
      \int_{\cT_{g,n}/T_{g,n}}\Ch(T^{-\ell}_{v}(\cT_{g,n}))\Td(T_{v}\cT_{g,n}) 
+\frac{n}{2}\sign(\frac{1}{2}-\ell) \\
-\sum_{i=1}^{n}
\lrpar{ \frac{1}{2\tanh \lrpar{\frac{e_{i}}{2}}}-\frac{1}{e_{i}}}
-\lrpar{\frac{1}{2\pi \sqrt{-1}}}^{\frac{N}{2}}d\int_{0}^{\infty} \Str\lrpar{ \frac{\pa \bbA^{t}_{D_{\ell}}}{\pa t} e^{-(\bbA^{t}_{D_{\ell}})^{2}}} dt,
\label{int.7}\end{multline}  
where $\bbA^{t}_{D_{\ell}}$ is the Bismut superconnection and $N$ is the number operator
in $\Lambda T_{g,n}$.
To define the form $e_{i}$, let $\cL_{i}\to T_{g,n}$ be the complex line bundle which at
$[\Sigma]\in T_{g,n}$ is given by the restriction of $K_{v}$ at the $i$th puncture (marked point) of $\Sigma:= p^{-1}([\Sigma])$.  Then $e_{i}$ is the Chern form of $\cL_{i}$ as defined
by Wolpert \cite{Wolpert2}.  

Since the Teichm\"uller space $T_{g,n}$ is contractible, formula~\eqref{int.7} only contains
cohomological information in its degree zero part.  However, since it is local and each 
of its terms is invariant under the action of the Teichm\"uller modular group $\Mod_{g,n}$,
formula~\eqref{int.7} also holds on the moduli space $\cM_{g,n}= T_{g,n}/ \Mod_{g,n}$ in the
sense of orbifolds, where the fibration $p:\cT_{g,n}\to T_{g,n}$ is replaced
by the forgetful map $\pi_{n+1}: \cM_{g,n+1}\to \cM_{g,n}$ and where it acquires a topological meaning in higher degrees (see Corollary \ref{cor:Moduli}).  For instance, on the moduli space $\cM_{g,n}$, the Chern form $e_{i}$ represents the Miller class 
$\psi_{i}= c_{1}(\cL_{i})$, while the first term on the left-hand side of \eqref{int.7}  represents a linear combination of the Mumford-Morita classes 
\begin{equation}
   \kappa_{j}= (\pi_{n+1})_{*}( c_{1}(\psi^{j+1}_{n+1})), \quad j\in \bbN_{0}.
\label{int.8}\end{equation}
This formula could be thought of as a local version of the Grothendieck-Riemann-Roch
theorem applied to the forgetful map $\pi_{n+1}: \overline{\cM}_{g,n+1}\to \overline{\cM}_{g,n}$ and a certain sheaf on
$\overline{\cM}_{g,n+1}$ depending on $\ell$ (when $\ell=0$, it is the sheaf of sections of the trivial line bundle).  When $\ell=0$ or $\ell=1$, our formula agrees modulo boundary terms
with the one obtained by Bini \cite{Bini} using the Grothendieck-Riemann-Roch theorem.   

Our results should be compared with the result of Takhtajan and Zograf \cite{TZ} and 
Wolpert \cite{Wolpert2}, who gave the two form part of \eqref{int.7} by interpreting
it as the first Chern form of the corresponding determinant line bundle equipped with the 
Quillen connection.  
As in the compact case, the definition of the Quillen connection makes use of the determinant of the Laplacian.
However the presence of cusps induces continuous spectrum for the Laplacian and the usual definition of its determinant via zeta-regularization is necessarily more delicate.
Takhtajan and Zograf sidestepped this issue by defining the determinant in terms of the Selberg zeta function, in analogy with the compact case \cite{DP,Sarnak}.
The precise description of the heat kernel in \cite{Vaillant} allows us to proceed along the lines of 
\cite{APSbook, Piazza, BJP} and extend the zeta function definition to this context via renormalization.
Unlike previous efforts (see, e.g.,
\cite{Efrat1}, \cite{Efrat2} and \cite{Muller}) this definition does not make use of the hyperbolic structure of the underlying manifold and works more generally for the metrics considered in \cite{Vaillant, Albin-Rochon}. Furthermore we show that, for hyperbolic metrics on surfaces with cusps, the resulting zeta-regularized determinant coincides with that defined using the Selberg zeta function up to a universal constant
(see theorem~\ref{as.11} and corollary~\ref{lge.9})
\begin{equation}
\det'(\Delta_{\ell}) = \left\{ \begin{array}{ll}
         \alpha_{\ell,g,n} Z_{\Sigma}(\ell), &   \ell\ge 2;\\
         \alpha_{\ell,g,n}Z_{\Sigma}'(1), & \ell=0,1;
   \end{array}\right.
\label{int.9}\end{equation}
when $\ell\ge 0$, 
where $\alpha_{\ell,g,n}$ is a constant only depending on $\ell$, $g$ and $n$.
  With this determinant and thanks to the fact $\ker\db_{\ell}$ is a
holomorphic vector bundle on $T_{g,n}$,
 the construction of the Quillen connection and the computation of the 
 curvature are essentially as in \cite{BGV}, \cite{BF} with only minor changes.  In this way, 
 we recover
 the index formula of \cite{TZ} (see also \cite{Weng} and \cite{Wolpert2}), 
 \begin{equation*}
 \begin{split}
\frac{\sqrt{-1}}{2\pi}(\nabla^{Q_{\ell}})^{2}&= \frac{1}{2\pi i} \lrpar{
                   \int_{\cT_{g,n}/T_{g,n}} \Ch\lrpar{ T^{-\ell}(\cT_{g,n}/T_{g,n})} \cdot
                          \Td\lrpar{T\lrpar{\cT_{g,n}/T_{g,n}}} }_{[2]}  \\ 
                & -\sum_{i=1}^{n} \frac{e_{i}}{12},
\end{split}
 \label{int.10}\end{equation*}
 see theorem~\ref{qc.25} and corollary~\ref{qc.46} below.   


Our approach and the one of Takhtajan and Zograf \cite{TZ} use substantially the 
fact that the dimension of the kernel of the family $\db_{\ell}$ does not jump, so that these kernels fit together into a 
vector bundle on the Teichm\"uller space.  More generally, one can ask if 
the work of Bismut, Gillet and Soul\'e \cite{BGS1, BGS2, BGS3} for the 
determinant of $\db$-operators arising on (compact) K\"ahler fibrations could 
be adapted to non-compact situations in order to deal with examples where
the rank of the kernel jumps.

The paper is organized as follows.  In \S~\ref{hco.0}, we review the definition and main properties of hyperbolic cusp operators.  In \S~\ref{bc.0}, we explain the passage from a punctured Riemann surface to
a compact Riemann surface with boundary.  In \S~\ref{rr.0}, we describe how the $\db$-operator on
a punctured Riemann surface can be seen as a Dirac type hyperbolic cusp operator.  We also check
that Vaillant's formula~\eqref{int.5} agrees with the Riemann-Roch theorem in this case.  In
\S~\ref{ts.0} and \S~\ref{con.0}, we make a quick review of Teichm\"uller theory from our perspective.  We then
obtain our main result in \S~\ref{lf.0} by computing the eta forms appearing in the family
index formula of \cite{Albin-Rochon}. We also compare our formula with
the Grothendieck-Riemann-Roch theorem.  In \S~\ref{Selberg.0}, we study the determinant of various Laplacians on Riemann surfaces of finite area and relate them to Selberg's zeta function
following \cite{BJP}.  Finally, in \S~\ref{qc.0}, we adapt the standard computation of 
the curvature of the Quillen connection to our 
context and compare our result with those of Takhtajan-Zograf \cite{TZ},
Weng \cite{Weng} and Wolpert \cite{Wolpert2}.  

\begin{acknowledgement}
We would like to thank Leon Takhtajan and Peter Zograf for explaining to us their 
results.  We are also grateful to Rafe Mazzeo, Richard Melrose, Gabriele Mondello and Sergiu Moroianu for helpful conversations.
\end{acknowledgement}

\numberwithin{equation}{section}
\numberwithin{corollary}{section}

\section{Hyperbolic cusp operators}\label{hco.0}

Let $X$ be a smooth compact manifold with boundary $\pa X\ne \emptyset$.  Let $x\in\CI(X)$ be a 
boundary defining function, that is, $x$ is a positive function in the interior vanishing on
the boundary such that its differential $dx$ is nowhere zero on $\pa X$.  For $\epsilon >0$ sufficiently small, there is an induced collar neighborhood of $\pa X$ in $X$,
\begin{equation}
   c:\pa X\times [0,\epsilon)_{x}\to N_{\epsilon}:= \{p\in X\quad | \quad x(p)<\epsilon\}
   \subset X.
\label{hco.1}\end{equation}
Consider a Riemannian metric $g_{\hc}$ in the interior $X\setminus \pa X$ taking the form
\begin{equation}
    c^{*}g_{\hc}= \frac{dx^{2}}{x^{2}}+ x^{2}\pi_{L}^{*}g_{\pa X}
\label{hco.2}\end{equation}
in the collar neighborhood \eqref{hco.1}, where $g_{\pa X}$ is a Riemannian metric
on $\pa X$ and $\pi_{L}:\pa X\times [0,\epsilon)_{x}\to \pa X$  is the projection on the left factor.  Such a metric is called a \textbf{ product hyperbolic cusp metric}
(or product d-metric in the terminology of Vaillant \cite{Vaillant}).
This is a complete metric on the interior of $X$, hence the boundary $\pa X$ is at infinity.
Notice however that the volume of $X$ is finite with respect to the metric $g_{\hc}$.  Following the philosophy of Melrose, one can get operators that are adapted to this geometry at
infinity by considering the space of \textbf{hyperbolic cusp vector fields} $\cV_{\hc}(X)$, that is, the space of smooth vector fields on $X$ with length uniformly bounded with respect
to the metric $g_{\hc}$,
\begin{multline}
    \cV_{\hc}(X):= \{\xi\in \CI(X;TX) \quad | \quad \exists \, c>0 \;\mbox{such that} \\ 
             g_{\hc}(\xi(p),\xi(p))<c \; \forall \,p\in X\setminus \pa X  \}.
\label{hco.3}\end{multline}
If $z=(z_{1},\ldots,z_{n-1})$ are local coordinates on $\pa X$, then in the collar neighborhood
\eqref{hco.1}, a hyperbolic cusp vector field $\xi$ takes the form
\begin{equation}
  \xi= a x\frac{\pa}{\pa x} + \sum_{i=1}^{n-1}\frac{b_{i}}{x}\frac{\pa}{\pa z_{i}}
\label{hco.4}\end{equation}
where $a, b_{1},\ldots,b_{n-1}$ are smooth functions on $X$.  It is possible to define 
a vector bundle ${}^{\hc}TX$ on $X$ in such a way that its space of smooth sections
is canonically identified with hyperbolic cusp vector fields,
\begin{equation}
    \CI(X;{}^{\hc}TX)= \cV_{\hc}(X).
\label{hco.5}\end{equation}
In the interior $X\setminus \pa X$, the vector bundle ${}^{\hc}TX$ is isomorphic to the tangent bundle $TX$.  This identification does not extend to an isomorphism on the boundary of $X$.
The metric $g_{\hc}$ naturally induces a metric on ${}^{\hc}TX$ which is also well-defined
on the boundary.  

A quick check indicates that $\cV_{\hc}(X)$ is not closed under the Lie bracket.  To define higher order hyperbolic cusp operators, it is convenient to consider the conformally related metric
\begin{equation}
  g_{\cusp}:= \frac{1}{x^{2}} g_{\hc}.  
\label{hco.6}\end{equation}
The metric $g_{\cusp}$ is called a \textbf{product cusp metric}.  One can consider the corresponding \textbf{cusp vector fields}
\begin{multline}
    \cV_{\cusp}(X):= x\cV_{\hc}(X)=\{\xi\in \CI(X;TX) \quad | \quad \exists \, c>0 \;\mbox{such that} \\ 
             g_{\cusp}(\xi(p),\xi(p))<c \; \forall \,p\in X\setminus \pa X  \}.
\label{hco.7}\end{multline}
Alternatively, one can define cusp vector fields by 
\begin{equation}
  \cV_{\cusp}(X):= \{ \xi\in \CI(X;TX) \quad | \quad \xi x\in x^{2}\CI(X)\},
\label{hco.7b}\end{equation}
which makes it clear that the definition only depends on the choice of boundary defining
function $x$ and not on the choice of metric $g_{\cusp}$.
There is also an associated vector bundle ${}^{\cusp}TX$ over $X$ whose space of smooth 
sections is canonically identified with the space of cusp vector fields,
\begin{equation}
                \CI(X;{}^{\cusp}TX)= \cV_{\cusp}(X).
\label{hco.8}\end{equation}
In the collar neighborhood \eqref{hco.1}, a cusp vector field $\xi$ has to be of the form
\begin{equation}
   \xi= a x^{2}\frac{\pa}{\pa x}+ \sum_{i=1}^{n-1} b_{i}\frac{\pa}{\pa z_{i}}
\label{hco.9}\end{equation}
with $a, b_{1},\ldots,b_{n-1}\in \CI(X)$.  As opposed to $\cV_{\hc}(X)$, the space
$\cV_{\cusp}(X)$ is closed under the Lie bracket, so that it is naturally a Lie algebra.  Its corresponding universal enveloping algebra is the space $\Diff^{*}_{\cusp}(X)$ of 
\textbf{cusp differential operators}.  In the collar neighborhood \eqref{hco.1}, a cusp
differential operator of order $k$, $P\in \Diff^{k}_{\cusp}(X)$, takes the form
\begin{equation}
   P= \sum_{l+|\alpha|\le k} p_{l,\alpha} \lrpar{x^{2}\frac{\pa}{\pa x}}^{l}
       \lrpar{\frac{\pa}{\pa z}}^{\alpha}, \quad p_{l,\alpha}\in \CI(X).
\label{hco.10}\end{equation}
More generally, Mazzeo and Melrose in \cite{MazzeoMelrose} defined the space
of \textbf{cusp pseudodifferential operators of order k}, $\Psi^{k}_{\cusp}(X)$.  These operators
are closed under composition,
\begin{equation}
   \Psi^{k}_{\cusp}(X)\circ \Psi^{l}_{\cusp}(X)\subset \Psi^{k+l}_{\cusp}(X).
\label{hco.11}\end{equation}
There is a corresponding \textbf{cusp Sobolev space of order} $m\in \bbN_{0}$, 
\begin{equation}
  H^{m}_{\cusp}(X):= \{ f\in L^{2}_{g_{\cusp}}(X)\, | \,
                  Pf\in L^{2}_{g_{\cusp}}(X) \; \forall \ P\in \Psi^{m}_{\cusp}(X)\}.
\label{hco.12}\end{equation}
One can also consider its weighted version $x^{k} H^{m}_{\cusp}(X)$ by some power
$x^{k}$ of the boundary defining function.  A cusp pseudodifferential operator 
$P\in \Psi^{m}_{\cusp}(X)$ then defines a bounded linear map
\begin{equation}
   P: x^{k}H^{l}_{\cusp}(X)\to x^{k}H^{l-m}_{\cusp}(X).
\label{hco.13}\end{equation} 
One interesting feature of the cusp operators is that if a cusp pseudodifferential $P\in \Psi^{m}(X)$ is invertible as a bounded linear map \eqref{hco.13}, then its inverse
is given by a cusp operator of order $-m$.  

Generalizing the relation $\cV_{\hc}(X)= \frac{1}{x} \cV_{\cusp}(X)$, one can define 
the space of \textbf{hyperbolic cusp pseudodifferential operators of order $m$}
by 
\begin{equation}
    \Psi^{m}_{\hc}(X):= x^{-m}\Psi^{m}_{\cusp}(X).
\label{hco.14}\end{equation}   
A hyperbolic cusp operator $P\in \Psi^{m}_{\hc}(X)$ naturally induces a bounded linear map
\begin{equation}
   P: x^{k}H^{l}_{\cusp}(X)\to x^{k-m}H^{l-m}_{\cusp}(X).
\label{hco.15}\end{equation}

So far we have considered operators acting on functions on $X$, but if $E\to X$ and $F\to X$ are
complex vector bundles on $X$, it is no more difficult to define the space of 
hyperbolic cusp operators $\Psi^{*}_{\hc}(X;E,F)$ acting from sections of $E$ to sections
of $F$.  

In \cite{MazzeoMelrose}, Mazzeo and Melrose gave a very elegant criterion to determine
when a cusp operator is Fredholm.  They first introduced a notion of principal symbol adapted
to the geometry at infinity, that is, involving the cosphere bundle $S^{*}({}^{\cusp}TX)$ of
${}^{\cusp}TX$,
\begin{equation}
   \sigma_{k}: \Psi^{k}_{\cusp}(X;E,F)\to \CI(S^{*}({}^{\cusp}TX); \hom(\pi^{*}E,\pi^{*}F))
\label{fc.1}\end{equation}
where $\pi: S^{*}{}^{\cusp}TX \to X$ is the bundle projection.  A cusp operator 
$A\in \Psi^{k}_{\cusp}(X;E,F)$ is said to be \textbf{elliptic} if its principal symbol 
$\sigma_{k}(A)$ is invertible.  In that case, by a standard construction, one can obtain
a parametrix $B\in \Psi^{k}_{\cusp}(X;F,E)$ such that 
\begin{equation}
  BA-\Id_{E}\in \Psi^{-\infty}_{\cusp}(X;E), \quad AB-\Id_{F}\in \Psi^{-\infty}_{\cusp}(X;F).
\label{fc.2}\end{equation}
However, since elements of $\Psi^{-\infty}_{\cusp}(X;E)$ are not compact in general, this does
not insure that the operator $A$ is Fredholm.  One needs some extra decay at infinity for the
error term to be compact.  Precisely, the subset of compact operators in $\Psi^{-\infty}(X;E)$ is
given by $x\Psi^{-\infty}(X;E)$.  It is possible to insure the error term is in that subset
provided $A$ is `invertible at infinity'.  This condition is determined by the normal operator
map
\begin{equation}
   N: \Psi^{k}_{\cusp}(X;E,F)\to \Psi^{k}_{\sus}(\pa X;E,F)
\label{fc.3}\end{equation}
where $\Psi^{k}_{\sus}(\pa X;E,F)$ is the space of suspended operators of order $k$ introduced
by Melrose in \cite{Melrose_eta}.  These are operators on $\pa X\times \bbR$ which are translation invariant in the $\bbR$ direction.  Essentially, the normal operator $N(A)$ of $A$
is its asymptotically translation invariant part at infinity.  The criterion of Mazzeo and Melrose can now be stated as follows.
\begin{proposition}[Mazzeo-Melrose]
A cusp operator $A\in \Psi^{k}_{\cusp}(X;E,F)$ is Fredholm if and only if it is elliptic and its
normal operator $N(A)$ is invertible.
\label{fc.4}\end{proposition}

For hyperbolic cusp operators, the situation is much more delicate.  For simplicity, let us restrict to a first order hyperbolic cusp \textbf{differential} operator $\eth_{\hc}\in
\Psi^{1}_{\hc}(X;E,F)$.  Then $x\eth_{\hc}\in \Psi^{1}_{\cusp}(X;E,F)$ is a cusp operator and we
can use proposition~\ref{fc.4} to determine whether or not $x\eth_{\hc}$ is Fredholm.  If it is
Fredholm, then it is not hard to see that $\eth_{\hc}$ is Fredholm as well.  In fact, in that case, the spectrum of $\eth_{\hc}$ is then necessarily discrete since its parametrix in
$x\Psi^{-1}_{\cusp}(X;F,E)$ is a compact operator.  

However, even if $x\eth_{\hc}$ is not Fredholm, it is still possible for $\eth_{\hc}$ to 
be Fredholm.  Define the \textbf{vertical family} of $\eth_{\hc}$ to be
\begin{equation}
   \eth_{\hc}^{V}:= \left.\lrpar{ x\eth_{\hc}}\right|_{\pa X}\in \Psi^{1}(\pa X;E,F).
\label{fc.5}\end{equation}
When $\eth_{\hc}$ is a self-adjoint Dirac type operator with $E=F$ a Clifford bundle, the
vertical family $\eth_{\hc}^{V}$ is invertible if and only if the normal operator 
$N(x\eth_{\hc})$ is invertible.  In his thesis \cite{Vaillant}, Vaillant gave the following 
criterion to determine if a Dirac-type self-adjoint operator $\eth_{\hc}$ is Fredholm.
The vertical family does not have to be invertible, but if it is not, Vaillant defined another
operator $\eth^{H}_{\hc}$ acting on the finite dimensional vector space $\mathcal{K}:= \ker \eth^{V}_{\hc}$ and called the \textbf{horizontal family}.  If $\Pi_{0}$ denotes the projection
from $L^{2}(\pa X;E)$ onto $\mathcal{K}$, then the horizontal family is defined by extending an element $\xi\in \mathcal{K}$ into the interior to an element $\widetilde{\xi}\in
\CI(X;E)$ and then applying $\eth_{\hc}$ and $\Pi_{0}$,
\begin{equation}
    \eth^{H}_{\hc}\xi:= \Pi_{0}\lrpar{ \left.\eth_{\hc} \widetilde{\xi}\right|_{\pa X}}.
\label{fc.6}\end{equation} 
In his thesis \cite{Vaillant}, Vaillant gave the following criterion.
\begin{proposition}[Vaillant \cite{Vaillant}, \S3]
A Dirac type self-adjoint operator $\eth_{\hc} \in \Psi^{1}_{\hc}(X;E)$ is Fredholm if and only
if $\eth^{H}_{\hc}$ is invertible.  Moreover, the continuous spectrum of $\eth_{\hc}$ is governed
by $\eth^{H}_{\hc}$ with bands of continuous spectrum starting at the eigenvalues of $\eth^{H}_{\hc}$ and going to infinity.
\label{FV}\end{proposition}

\section{The boundary compactification of a Riemann surface with puncture}
\label{bc.0}

Let $\Sigma$ be a Riemann surface of type $(g,n)$, that is,
$\Sigma= \overline{\Sigma}\setminus\{x_{1},\ldots,x_{n}\}$ where
$\overline{\Sigma}$ is a compact Riemann surface of genus $g$ and 
$x_{1},\ldots,x_{n}$ are pairwise distinct points on $\overline{\Sigma}$.
We will assume that $2g+n\ge 3$.  The surface $\overline{\Sigma}$ is a compactification of $\Sigma$.  An 
alternative way of compactifying the Riemann surface $\Sigma$ is to consider
the radial blow up $\Sigma_{b}$ of $\overline{\Sigma}$ at the points
$\{x_{1},\ldots,x_{n}\}$ with blow-down map
\begin{equation}
   \beta: \Sigma_{b}\to \overline{\Sigma}.
\label{bc.1}\end{equation}
This gives a compactification of $\Sigma$ in which each puncture is replaced by a circular boundary.  
The Riemann surface with boundary $\Sigma_{b}$ also
comes equipped with a natural choice of boundary defining function
$\rho\in \CI(\Sigma_{b})$ as we will see.  This choice is dictated by the
uniformization theorem for Riemann surfaces.  

Recall that, by the uniformization theorem, there is a canonical
hyperbolic metric $g_{\Sigma}$ on $\Sigma$ obtained by taking the 
unique metric of constant scalar curvature equal to $-1$ in the conformal 
class defined by the complex structure of $\Sigma$.  Consider the upper-half
plane 
\begin{equation}
\bbH= \{ x+iy \in \bbC \; | \; y>0\}
\label{bc.2}\end{equation}
equipped with the Poincar\'e metric
\begin{equation}
   g_{\bbH}:= \frac{dx^{2}+ dy^{2}}{y^{2}}.
\label{bc.3}\end{equation}
Let $\Gamma_{\infty}$ be the discrete Abelian group generated by the 
parabolic isometry $z\mapsto z+1$.  The \emph{horn} is the quotient
\begin{equation}
   H:= \Gamma_{\infty} \setminus \bbH.
\label{bc.4}\end{equation}
Via the change of variable $r=\frac{1}{y}$, one sees that 
the horn is isometric to 
$(0,+\infty)_{r}\times \bbR/\bbZ$ equipped with the metric
\begin{equation}
    \frac{dr}{r^{2}}+ r^{2}dx^{2}.
\label{bc.5}\end{equation}
A \emph{cusp end} is a subspace of $H$ of the form $(0,a]\times
\bbR/\bbZ$.
Near a puncture $x_{i}$ of $\Sigma$, the geometry of $(\Sigma,g_{\Sigma})$
is modeled on a cusp end.  That is, around each puncture $x_{i}$, there 
exists a neighborhood $N_{i}\subset \Sigma$ and an isometry
\begin{equation}
  \varphi_{i}: N_{i}\to C_{i}
\label{bc.6}\end{equation}
with a cusp end $C_{i}=(0,\frac{1}{y_{i}}]\times \bbR/\bbZ$.
Each cusp end has a natural compactification
\begin{equation}
	\overline{C}_{i}= \left[ 0,\frac{1}{y_{i}} \right]_{r_{i}}\times \bbR/\bbZ
\label{bc.7}\end{equation}      
where the coordinate $r_{i}$ can be seen as a boundary defining function
for the boundary $\{0\}\times \bbR/\bbZ\subset \overline{C}_{i}$.  
This boundary defining function can in fact be defined intrinsically in
terms of the hyperbolic metric \eqref{bc.5}.  Indeed, we define a
\emph{horocycle} to be an embedded circle in a cusp end which is 
perpendicular to all geodesics emanating from the cusp.  This definition is 
formulated purely in terms of the metric.  On the other hand, as one can check,
the horocycles are precisely given by the level sets of the function $r_{i}$.  
Moreover, the value of the function $r_{i}$ on a horocycle 
$\gamma= \{u\}\times \bbR/\bbZ$ is also determined
by the hyperbolic metric.  It is the area of the smaller cusp end
$(0,u)\times \bbR/\bbZ$, namely 
\begin{equation}
   r_{i}(u,v)= \operatorname{area}( (0,u)\times \bbR/\bbZ))=
\int_{0}^{u}\int_{\bbR/\bbZ} dr dx= u.
\label{bc.8}\end{equation}
Thus, intuitively, the boundary defining function $r_{i}$ is the 
`area function' for the cusp end $C_{i}$.  The compactification
$\overline{C}_{i}$ induces a corresponding compactification 
$\overline{N}_{i}$ via the isometry \eqref{bc.6}, and thus a compactification
$\Sigma_{\hc}$ of $\Sigma$ into a compact surface with 
boundary naturally diffeomorphic to $\Sigma_{b}$.  To get
a global boundary defining function, choose 
a smooth non-decreasing function $\chi\in \CI([0,+\infty))$ such that 
\begin{equation}
   \chi(x):= \left\{ \begin{array}{ll}
                    x, & \mbox{if}\; 0\le x\le \frac{1}{2}; \\
                    1, & \mbox{if} \; x\ge 1,
\end{array}
\right. 
\label{bc.9}\end{equation}
and consider
 $\chi_{\epsilon}(x):= \epsilon \chi(\frac{x}{\epsilon})$ for 
$0<\epsilon< \min\{\frac{1}{y_{1}},\ldots,\frac{1}{y_{n}}\}$.
On each (compactified) cusp end $\overline{C}_{i}$, consider
the function $\chi_{\epsilon}(r_{i})$.  Then the function 
\begin{equation}
 \rho_{\Sigma,\epsilon}(\sigma):= \left\{ \begin{array}{ll}
                    \varphi_{i}^{*}(\chi_{\epsilon}\circ r_{i})(\sigma), & \mbox{if}\;
 \sigma\in \overline{N}_{i}, \quad i\in\{1,\ldots,n\}; \\
                    \epsilon, & \mbox{otherwise};
\end{array}
\right. 
\label{bc.10}\end{equation} 
is a boundary defining function for $\pa\Sigma_{\hc}$ in
$\Sigma_{\hc}$.  Since the choice of the number $\epsilon$ is 
not of primary importance, we will usually denote the function
$\rho_{\Sigma,\epsilon}$ simply by $\rho_{\Sigma}$.
With respect to this boundary defining function,
the hyperbolic metric $g_{\Sigma}$ is a product hyperbolic metric.
That is, in the coordinates $(x,\rho_{\Sigma})$ on $N_{i}$, it is of the form
\begin{equation}
g_{\Sigma}= \frac{d\rho_{\Sigma}^{2}}{\rho_{\Sigma}^{2}}+ 
\rho_{\Sigma}^{2}dx^{2}
\label{bc.11}\end{equation}
near the boundary.

\section{The $\db$-operator as a Dirac-type $\hc$-operator}\label{rr.0}

Let $K:= \Lambda_{\Sigma}^{1,0}$ denote the canonical line bundle on
$\Sigma$.  This line bundle and all of its tensor powers $K^{\ell}$ have natural holomorphic structures.
In particular, for each $\ell\in \bbZ$, there is a well-defined
$\db$ operator
\begin{equation}
  \db_{\ell}: \CI(\Sigma; K^{\ell})\to \CI(\Sigma; \Lambda^{0,1}_{\Sigma}\otimes K^{\ell}),
\label{bc.12}\end{equation}
where $\Lambda^{0,1}_{\Sigma}\to \Sigma$ is the bundle of $(0,1)$-forms
on $\Sigma$.  In a cusp end $C_{i}$ where the canonical line bundle is 
trivialized by the holomorphic section $dz$, it takes the form
\begin{equation}
\begin{aligned}
d\overline{z} \frac{\pa}{\pa \overline{z}}&= (dx-idy)\frac{1}{2}
\left( \frac{\pa}{\pa x} + i\frac{\pa}{\pa y}\right)  \\
     &=  (dx+ i\frac{dr}{r^{2}})\frac{1}{2}
   \left( \frac{\pa}{\pa x}- ir^{2}\frac{\pa}{\pa r}\right), \quad r= \frac{1}{y},\\
 &= \frac{1}{2}\left( r dx+ i \frac{dr}{r} \right) 
   \left( \frac{1}{r}\frac{\pa}{\pa x} -ir\frac{\pa}{\pa r} \right).
\end{aligned}
\label{bc.13}\end{equation}
Thus, near the boundary $\pa \Sigma_{\hc}$, the $\db$-operator is of the form
\begin{equation}
\db=\frac{1}{2}\left( \rho_{\Sigma} dx+ i \frac{d\rho_{\Sigma}}{\rho_{\Sigma}} \right) 
   \left( \frac{1}{\rho_{\Sigma}}\frac{\pa}{\pa x} -i\rho_{\Sigma}\frac{\pa}{\pa \rho_{\Sigma}} \right).
\label{bc.14}\end{equation}
Since $\frac{1}{\rho_{\Sigma}}\frac{\pa}{\pa x} -i \rho_{\Sigma}\frac{\pa}{\pa\rho_{\Sigma}}$
is a $\hc$-operator and 
$\frac{1}{2}(\rho_{\Sigma} dx+ i\frac{d\rho_{\Sigma}}{\rho_{\Sigma}})$ is naturally a section of ${}^{\hc}T^{*}\Sigma\otimes_{\bbR}\bbC$, we see that the $\db_{\ell}$-operator naturally
extends to give a $\hc$-operator 
\begin{equation}
 \db_{\ell}: \CI(\Sigma_{\hc}; {}^{\hc}K^{\ell})\to \frac{1}{\rho_{\Sigma}}\CI(\Sigma_{hc}; 
{}^{\hc}\Lambda^{0,1}_{\Sigma}\otimes {}^{\hc}K^{\ell})
\label{bc.15}\end{equation}
where  ${}^{\hc}\Lambda^{0,1}_{\Sigma}$ is the complex conjugate of ${}^{\hc}K$ and 
${}^{\hc}K\subset {}^{\hc}T^{*}\Sigma\otimes_{\bbR}\bbC $
is such that it is identified with $K$ in the interior of
$\Sigma_{\hc}$ and it is trivialized by the section 
$\rho_{\Sigma}dz=\rho_{\Sigma} dx- i\frac{d\rho_{\Sigma}}{\rho_{\Sigma}}$ 
near each connected component of the boundary.
The metric $g_{\Sigma}$ induces a Hermitian metric on
$K^{\ell}$ and $\Lambda^{0,1}_{\Sigma}$, as well as on
${}^{\hc}K^{\ell}$ and ${}^{\hc}\Lambda^{0,1}_{\Sigma}$.  We denote 
by $\cH_{\ell,i}$ the Hilbert space of square integrable sections
of ${}^{\hc}K^{\ell}\otimes ({}^{\hc}\Lambda^{0,1})^{i}$ with respect
to the natural scalar product
\begin{equation}
\langle f_{1}, f_{2}\rangle_{\cH_{\ell,i}}:=
\int_{\Sigma_{\hc}} \langle f_{1}(\sigma),f_{2}(\sigma)
\rangle_{g_{\Sigma}} 
dg_{\Sigma}(\sigma)
\label{rr.1}\end{equation}
where $dg_{\Sigma}$ is the natural extension of the volume form of $g_{\Sigma}$
on $\Sigma_{\hc}$.

The operator $\db_{\ell}$ is Fredholm.  To 
see this, recall (see for instance proposition 3.67 in \cite{BGV}) that 
\begin{equation}
 D_{\ell}:= \sqrt{2}( \db_{\ell}+ \db^{*}_{\ell})
\label{rr.5}\end{equation}
is a Dirac type operator induced by the Chern
connection on $K^{\ell}$ with Clifford action on
$\nu\in \CI(\Sigma_{\hc};{}^{\hc}\Lambda_{\Sigma})$ given by
\begin{equation}
  c(f)\nu= \sqrt{2}( \varepsilon(f^{0,1}) -\iota(f^{1,0}))\nu, \quad
f\in \CI(\Sigma_{\hc};{}^{\hc}\Lambda_{\Sigma}),
\label{rr.6}\end{equation}
where $\varepsilon(f^{0,1})$ denotes exterior multiplication by $f^{0,1}$.
The operator $D_{\ell}$ is formally self-adjoint.  The vertical family $D^{V}_{\ell}$ of $D_{\ell}$ is given
by 
\begin{equation}
  c(du) \frac{\pa}{\pa u}
\label{fp.1}\end{equation}
acting on $\CI(\bbR/ \bbZ;
    {}^{\hc}K^{\ell}\oplus {}^{\hc}\Lambda^{0,1}_{\Sigma}\otimes {}^{\hc}K^{\ell})$ on each circular boundary component of $\Sigma_{\hc}$, where $u=-x$ is such that 
$\{ \left. \frac{\pa}{\pa u}\right|_{\sigma}\}$ is an oriented orthonormal basis of $T_{\sigma}\pa \Sigma_{\hc}$ for each $\sigma\in \pa \Sigma_{\hc}$.  In particular, $\mathcal{K}= \ker(D^{V}_{\ell})$ is a complex vector space of dimension $2n$.  By proposition~\ref{FV}, we
need to show that the horizontal family $D^{H}_{\ell}: \mathcal{K}\to\mathcal{K}$ is invertible.
\begin{proposition}
On each circular boundary component of $\Sigma_{\hc}$, the 
horizontal family is given by 
\[
    D^{H}_{\ell}= \lrpar{\ell-\frac{1}{2}} ic(du).
\]
\label{fp.2}\end{proposition} 
\begin{proof}
The bundle on which $D_{\ell}$ acts is 
\[
     \lrpar{ \underline{\bbC}\oplus \Lambda^{0,1}\Sigma}\otimes K^{\ell}.
\]
Choose a spin structure on $\Sigma$ and let $S$ be the corresponding spinor bundle.  It is well-known (see for instance \cite{Lawson-Michelson}) that, seen as complex line bundle, $S$ is a square root of the canonical line bundle $K$ so that 
\[
      S\otimes_{\bbC} S= K.
\]
Moreover, we have also that 
\[
                 \lrpar{ \underline{\bbC}\oplus \Lambda^{0,1}\Sigma}\cong S \otimes_{\bbR} S^{*}.
\]
Thus the operator $D_{\ell}$ acts on
\[
     S\otimes_{\bbR} \lrpar{ S^{*}\otimes_{\bbC} K^{\ell}},
\]
which means that $D_{\ell}$ is a Dirac operator twisted by the bundle 
$S^{*}\otimes_{\bbC}K^{\ell}$.  As a bundle with connection, the bundle $S^{*}\otimes
K^{\ell}$ certainly does not have a product structure near the boundary since it has non-zero
curvature.  Thus, according to Proposition 3.15, p.44 in \cite{Vaillant}, the horizontal family
$D^{H}_{\ell}$ at each cusp is given by 
\[
    -iRc\lrpar{ \frac{\pa}{\pa u}} = -i\lrpar{\frac{1}{2}-\ell} c\lrpar{\frac{\pa}{\pa u}}
\]
where $iR dg_{\Sigma}$ is the curvature of the complex vector bundle $S^{*}\otimes_{\bbC} K^{\ell}$ (\cf \eqref{rr.13}).  Collecting the contributions at each cusp end, we get the desired result.
\end{proof}
This gives the following corollary.
\begin{corollary}
The operators 
\[
   D_{\ell}= \sqrt{2} \left( \begin{array}{cc}
                                 0 & \db^{*}_{\ell} \\
                                 \db_{\ell} & 0 \end{array}
   \right), \quad \db_{\ell}, \;\mbox{and}\;\;\db^{*}_{\ell}
\]
are Fredholm.
\label{fp.3}\end{corollary}
Notice that proposition~\ref{fp.2} is also consistent with the well-known fact that the
band of continuous spectrum of the Hodge Laplacian $D_{\ell}^{2}$ starts at 
$\lrpar{ \frac{1}{2}-\ell}^{2}$ and goes to infinity.

In his
thesis \cite{Vaillant}, Vaillant obtained a general formula for the index
of a Dirac type operator on a fibred hyperbolic cusp operator.  For the index
of the operator $\db_{\ell}$, this formula is given by the usual
Atiyah-Singer integrand together with two corrections coming from the boundary,
namely the eta invariants associated to the vertical family of $\db_{\ell}$ and the
horizontal family $D^{H}_{\ell}$,
\begin{equation}
\ind( \db_{\ell})=
\int_{\Sigma_{\hc}} \Ch({}^{\hc}K^{\ell})\Td({}^{\hc}K^{-1}) -
\frac{1}{2} \eta(D^{V}_{\ell}) - \frac{1}{2}\eta(D^{H}_{\ell}).
\label{rr.8}\end{equation}
The eta invariant of the vertical family  is easily seen to be zero.  This is because modulo standard identifications, $\eta(D^{V}_{\ell})$ corresponds to $n$ times the eta invariant
of the self-adjoint operator
\begin{equation}
   \frac{1}{i}\frac{\pa}{\pa x}= i\frac{\pa}{\pa u}: 
\CI( \bbR/\bbZ) \to \CI(\bbR/\bbZ).
\label{rr.9}\end{equation}
But the spectrum
of $\frac{1}{i}\frac{\pa}{\pa x}$ is  $2\pi\bbZ$ and its eta functional 
\begin{equation}
 \eta( \frac{1}{i}\frac{\pa}{\pa x}, s)= \sum_{k\ne 0} 2\pi k |2\pi k|^{-s}, \quad
\Re s>> 0
\label{rr.10}\end{equation}
 is identically zero.  Thus its spectral asymmetry or eta invariant, which
is the value at $s=0$ of the analytic continuation of 
$\eta( i\frac{\pa}{\pa x}, s)$, is zero.  The corresponding eta 
invariant $\eta(D_{\ell}^{V})= n\eta( i\frac{\pa}{\pa x})$ therefore
vanishes. For the computation of the spectral asymmetry of $D^{H}_{\ell}$, there is no 
regularization involved since $D^{H}_{\ell}$ is just an endomorphism of a finite dimensional
vector space.  From proposition~\ref{fp.2}, we compute directly (see \cite[(4.14)]{Albin-Rochon}) that
\begin{equation}
     \eta( D^{H}_{\ell}) = n \sign\lrpar{\ell-\frac{1}{2}}.
\label{fp.18}\end{equation}
The index is therefore given by
\begin{equation}
\ind( \db_{\ell})=
\int_{{}^{\hc}\Sigma} \Ch({}^{\hc}K^{\ell})\Td({}^{\hc}K^{-1}) +
\frac{n}{2}\sign\lrpar{\frac{1}{2}-\ell}.
\label{rr.11}\end{equation}
The integral is also easy to compute.  Let $\Theta_{\Sigma}$ denote the 
curvature of ${}^{\hc}T^{1,0}\Sigma$.  Then the integrand is
given by
\begin{equation}
\begin{aligned}
\Ch({}^{\hc}K^{\ell})\Td({}^{\hc}T^{1,0}\Sigma)&= 
\left(e^{-\frac{\ell i}{2\pi}\Theta_{\Sigma}}\right) 
\left( \frac{\frac{i}{2\pi}
\Theta_{\Sigma}}{1-e^{-\frac{i}{2\pi}\Theta_{\Sigma}}}\right) \\
&= 1 + \lrpar{\frac{1}{2}- \ell}\frac{i}{2\pi}\Theta_{\Sigma}.
\end{aligned}
\label{rr.12}\end{equation}
By a standard computation (see for instance p.77 in \cite{Griffiths-Harris}), we know that 
\begin{equation}
  \frac{i}{2\pi} \Theta_{\Sigma} = \frac{\kappa}{2\pi} dg_{\Sigma}= -\frac1{2\pi}dg_{\Sigma}
\label{rr.13}\end{equation}
where $\kappa=-1$ is the Gaussian curvature of $g_{\Sigma}$.  
By the Gauss-Bonnet
theorem applied to $\Sigma$, we get that 
\begin{equation}
\begin{aligned}
\ind( \db_{\ell}) &= \lrpar{\frac{1}{2}-\ell} \int_{\Sigma_{\hc}} \frac{\kappa}{2\pi}dg_{\Sigma} + \frac{n}{2}\sign\lrpar{\frac{1}{2}-\ell}  \\
   &= \lrpar{\frac{1}{2}-\ell} \chi(\Sigma)+ \frac{n}{2}\sign\lrpar{\frac{1}{2}-\ell}  \\
&= \lrpar{\frac{1}{2}-\ell}(2-2g-n) + \frac{n}{2}\sign\lrpar{\frac{1}{2}-\ell}.
\end{aligned}
\label{rr.14}\end{equation}
This gives the following formula.
\begin{proposition}
The index of $\db_{\ell}$ is given by
\begin{equation*}
   \ind(\db_{\ell})= \left\{  
      \begin{array}{ll}
         (2\ell-1)(g-1) + \ell n, & \ell\le 0,\\
         (2\ell-1)(g-1) + (\ell-1)n, & \ell>0.
      \end{array}
   \right.
\end{equation*}
\label{rrt}\end{proposition}
In fact, using the Riemann-Roch theorem on the compact Riemann surface $\overline{\Sigma}$, it is also possible to compute explicitly the dimension of the 
kernel and the 
cokernel of $\db_{\ell}$ (\cf p.404 in \cite{TZ}).  By definition,
an element of $f\in \ker{\db_{\ell}}$ is a holomorphic section of
$K^{\ell}$, so in each cusp end $N_{j}$, it has a Laurent series expansion
\begin{equation}
  f(z)= \sum_{k=-\infty}^{\infty} a^{(j)}_{k}e^{2\pi i kz} (dz)^{\ell}.
\label{rr.15}\end{equation}  
When $\ell>0$, this expansion has to be of the form 
\begin{equation}
  f(z)= \sum_{k=1}^{\infty} a^{(j)}_{k}e^{2\pi i kz} (dz)^{\ell}
\label{rr.16}\end{equation}
in order for $f$ to be an element of $\cH_{\ell,0}$.  Such an $f$ is 
said to be a \textbf{cusp form} of weight $(2\ell,0)$.  When $\ell\le 0$,
we can also have a constant coefficient in the series,
\begin{equation}
f(z)= \sum_{k=0}^{\infty} a^{(j)}_{k}e^{2\pi i kz}(dz)^{\ell}.
\label{rr.17}\end{equation}
When $\ell=0$, using the coordinate $\zeta:= e^{2\pi i z}$ near
each puncture $x_{j}$ in $\overline{\Sigma}$, we see that such a 
$f$ naturally extends to give a holomorphic function on 
$\overline{\Sigma}$.  It is therefore constant, so that
$\dim_{\bbC} \ker \db_{0}=1$.  When $\ell\ge 1$, the section $f$ takes the
form
\begin{equation}
 f(z)= \sum_{k=1}^{\infty} a^{(j)}_{k} \zeta^{k} \left( 
 \frac{d\zeta}{2\pi i \zeta}\right)^{\ell} 
\label{rr.19}\end{equation}
in the coordinate $\zeta$ near the puncture $x_{j}$.  Thus, it naturally
extends to a meromorphic section of $\overline{K}^{\ell}\to \overline{\Sigma}$
with poles of order not exceeding $\ell-1$ at each puncture 
$x_{1},\ldots,x_{n}$ and holomorphic elsewhere.  Conversely, such a meromorphic
section corresponds to an element of $\ker \db_{\ell}$.  We can thus
compute $\dim_{\bbC}\ker \db_{\ell}$ by applying the Riemann-Roch theorem
on $\overline{\Sigma}$ to the line bundle
\begin{equation}
  L_{D}\otimes \overline{K}^{\ell}
\label{rr.20}\end{equation}
where $L_{D}$ is the holomorphic line bundle associated to the divisor
\begin{equation}
     D= \sum_{i=1}^{n} (\ell-1)x_{i} \quad \mbox{on} \;\overline{\Sigma}.
\label{rr.21}\end{equation}
This gives
\begin{equation}
\begin{aligned}
\dim_{\bbC} \ker \db_{\ell}&= h^{0}(L_{D}\otimes \overline{K}^{\ell}) \\
 &= h^{0}(\overline{K}\otimes (L_{D}\otimes \overline{K}^{\ell})^{-1})
        +\deg(L_{D}\otimes \overline{K}^{\ell})-g +1 \\
&= h^{0}(\overline{K}\otimes (L_{D}\otimes \overline{K}^{\ell})^{-1})+
n(\ell-1)+ (2\ell-1)(g-1),
\end{aligned}
\label{rr.22}\end{equation}
where $h^{0}(L)$ denotes the dimension of the space of holomorphic sections
of the holomorphic line bundle $L$.
Now we compute that 
\begin{equation}
  \deg(\overline{K}\otimes L_{D}^{-1}\otimes \overline{K}^{-\ell})= -(\ell-1)(2g+n-2).
\label{rr.23}\end{equation}
When $\ell=1$, $\overline{K}\otimes 
(L_{D}\otimes \overline{K}^{\ell})^{-1}$ is the trivial line bundle,
so $h^{0}(\overline{K}\otimes 
(L_{D}\otimes \overline{K}^{\ell})^{-1})=1$ in this case.  When $\ell>1$,
$\deg(K\otimes L_{D}^{-1}\otimes K^{-\ell})<0$ since we assume that
$2g+n\ge 3$, and therefore $h^{0}(K\otimes L_{D}^{-1}\otimes K^{-\ell})=0$.
Finally, when $\ell<0$, elements of $\ker\db_{\ell}$ correspond to
holomorphic sections of $K^{\ell}_{\overline{\Sigma}}$ with zeros
of degree at least $-\ell$ at each puncture.  These in turn correspond
to the holomorphic sections of a holomorphic line bundle of negative degree
(since $2g+n\ge 3$), so that $\ker\db_{\ell}=0$ in that case.
Hence, we see that the dimension of the kernel of $\db_{\ell}$ is given
by
\begin{equation}
\dim \ker \db_{\ell} = \left\{ \begin{array}{ll}
                       0, & \ell<0, \\
                       1, & \ell=0, \\
                       g, & \ell=1 \\
                       (2\ell-1)(g-1)+ n(\ell-1), & l\ge 2.
\end{array}
\right.
\label{rr.24}\end{equation}
Comparing with the index \eqref{rr.14}, we also get that
\begin{equation}
\dim \ker \db_{\ell}^{*} = \left\{ \begin{array}{ll}
                       -(2\ell-1)(g-1)- n\ell, & \ell<0, \\
                       g, & \ell=0, \\
                       1, & \ell=1 \\
                       0, & l\ge 2.
\end{array}
\right.
\label{rr.25}\end{equation}
These formulas are consistent with Kodaira-Serre duality, which asserts in this case
that $\ker \db^{*}_{\ell}\cong \ker \db_{1-\ell}$.

\section{The Teichm\"uller space and the Teichm\"uller universal curve}
\label{ts.0}

So far we have assumed that the complex structure on $\Sigma$ was fixed.  By
changing the complex structure, one can get instead a family of $\db_{\ell}$ 
operators.  The universal case is obtained by considering all at once 
the moduli space of all complex structures on a surface of type $(g,n)$, two complex
structures being identified whenever there is a conformal transformation between them
homotopic to the identity.
It is called the Teichm\"uller space of Riemann surfaces of genus $g$ with
$n$ punctures and is denoted $T_{g,n}$.  It is a complex manifold of
complex  dimension $3g-3+n$ which can be identified with an open set
of $\bbC^{3g-3+n}$.  The Teichm\"uller space $T_{g,n}$ comes together
with a universal bundle, the universal Teichm\"uller curve $\cT_{g,n}$
with bundle projection
\begin{equation}
   p: \cT_{g,n}\to T_{g,n}
\label{ts.1}\end{equation} 
and fibre $p^{-1}([\Sigma])$ the Riemann surface $\Sigma$ of type $(g,n)$ 
corresponding to the point $[\Sigma]\in T_{g,n}$.  Denote
by $T_{v}^{i,j}\cT_{g,n}\to \cT_{g,n}$ the vertical $(i,j)$ tangent bundle of the
fibration \eqref{ts.1} for $i,j\in\{0,1\}$.  On each fibre $\Sigma:= p^{-1}([\Sigma])$, the 
restriction of $T_{v}^{i,j}\cT_{g,n}$ is canonically identified
with $T^{i,j}\Sigma$. Denote by $\Lambda_{v}^{i,j}\to \cT_{g,n}$ the
dual of $T_{v}^{i,j}$.  
On each fibre we also have a $\db$-operator.
These operators fit together to give a family of operators
\begin{equation}
   \db_{\ell}\in \rho^{-1}\Psi^{1}_{\cusp}(\cT_{g,n}/T_{g,n};
             (\Lambda^{1,0}_{v})^{\ell}, \Lambda^{0,1}_{v}\otimes
                 (\Lambda_{v}^{1,0})^{\ell})
\label{ts.2}\end{equation}
where $\rho$ is an appropriate boundary defining function (whose precise definition we postpone to \eqref{con.19}). 
Each element of the family is a Fredholm operator so that we have a 
family index in $K^{0}(T_{g,n})$,
\begin{equation}
 \ind(\db_{\ell})\in K^{0}(T_{g,n}).
\label{ts.3}\end{equation}
Since the Teichm\"uller space is contractible, this families index really
only encodes the numerical index of any member of the family
under the identification $K^{0}(T_{g,n})\cong K^{0}(\pt)\cong \bbZ$.  Still,
it is possible to exhibit an explicit representative of the $K$-class
$\ind(\db_{\ell})\in K^{0}(T_{g,n})$, providing in this way a local description
of the family index.  This is because, according to \eqref{rr.24} and
\eqref{rr.25}, the dimensions of the kernel and the cokernel of elements
of the family $\db_{\ell}$ are always the same (they only depend on
$\ell$, $g$ and $n$, not on the complex structure).  This means that 
\begin{equation}
  \ker \db_{\ell}\to T_{g,n} \quad \mbox{and}\quad 
 \ker\db_{\ell}^{*}\to T_{g,n}
\label{ts.5}\end{equation}
form complex vector bundles on $T_{g,n}$ and the family index of 
$\db_{\ell}$ can then be expressed as the virtual difference of these two
vector bundles,
\begin{equation}
  \ind \db_{\ell}= [\ker \db_{\ell}]- [ \ker \db_{\ell}^{*}]\in 
K^{0}(T_{g,n}).
\label{ts.6}\end{equation} 
In fact, as we will recall in a moment, these vector bundles both come
equipped with a natural connection.  We can therefore express their
respective Chern characters at the level of forms.  This provides a 
local description of the Chern character of the family index
\begin{equation}
   \Ch( \ind \db_{\ell}):= \Ch(\ker \db_{\ell})-
                   \Ch(\ker \db_{\ell}^{*}) \in \CI(T_{g,n},
\Lambda^{\even}(T_{g,n})).
\label{ts.7}\end{equation}
On the Teichm\"uller space itself, this local description of the index
does not contain more cohomological information than \eqref{rr.14}.  However,
the local descriptions \eqref{ts.6} and \eqref{ts.7} are invariant
under the action of the Teichm\"uller modular group $\Mod_{g,n}$.  This means that
these local descriptions descend to the moduli space $T_{g,n}/\Mod_{g,n}$ (in the sense of orbifolds), which
typically has a non-trivial topology as well as singularities.

\section{The canonical connection on the universal Teichm\"uller curve}
\label{con.0}

The fibration $p:\cT_{g,n}\to T_{g,n}$ comes together with a canonical
connection $\cP$.  To describe this connection, one possible approach is
to describe Riemann surfaces as certain quotients of the upper 
half-plane $\bbH$.  If $\Sigma$ is a Riemann surface of genus $g$ with
$n$ punctures, then it can be represented as a quotient $\Gamma\setminus
\bbH$ of the upper half-plane by the action of a torsion-free finitely
generated Fuchsian group $\Gamma$.  The group $\Gamma\subset \PSL(2,\bbR)$
is of type $(g,n)$, which is to say it
is generated by $2g$ hyperbolic transformations 
$A_{1},B_{1},\ldots, A_{g},B_{g}$ and $n$ parabolic transformations
$S_{1},\ldots,S_{n}$ satisfying the single relation
$A_{1}B_{1}A_{1}^{-1}B_{1}^{-1}\cdots A_{g}B_{g}A_{g}^{-1}B_{g}^{-1}S_{1}
\cdots S_{n}=1$.  Since $\bbH$ is simply connected, in
fact contractible, it is the universal cover of $\Sigma$ under the quotient
map $\bbH\to \Gamma\setminus \bbH$ .
From this perspective, the canonical hyperbolic metric
$g_{\Sigma}$ associated to the (conformal structure of the) complex structure
is precisely the metric on $\Gamma\setminus \bbH$ induced from 
the Poincar\'e metric 
\begin{equation}
  g_{\bbH}:= \frac{dx^{2}+dy^{2}}{y^{2}}\; \mbox{on}\; \bbH.
\label{con.1}\end{equation}
The punctures of $\Sigma$ then correspond to the image of 
the fixed points $z_{1},\cdots,z_{n}$ in  $\bbR\cup\{\infty\}$ of the
parabolic transformations $S_{1},\ldots,S_{n}$ under the quotient
map $\bbH\to \Gamma\setminus \bbH$.  Let $\Gamma_{i}$ be the 
cyclic subgroup of $\Gamma$ generated by the parabolic transformation
$S_{i}$ for $i=1,\ldots,n$.  It can be identified with the cyclic
group $\Gamma_{\infty}$ by choosing $\sigma_{i}\in \PSL(2,\bbR)$ such that
$\sigma_{i}\infty= z_{i}$, so that 
\begin{equation}
   \sigma_{i}^{-1}S_{i}\sigma_{i}= \left(
                 \begin{array}{cc}
                        1& \pm 1 \\
                        0 & 1
                 \end{array}
\right), \quad 
                \sigma_{i}^{-1}\Gamma_{i}\sigma_{i}= \Gamma_{\infty}.
\label{con.2}\end{equation}
On $\Sigma$, sections of $(\Lambda^{1,0}_{\Sigma})^{\ell}\otimes
((\Lambda_{\Sigma})^{0,1})^{m}$ correspond to automorphic forms 
of weight $(2\ell,2m)$ with respect to the group $\Gamma$, that is, functions
$f:\bbH\to \bbC$ such that 
\begin{equation}
 f(\gamma z) \gamma'(z)^{\ell} \overline{ \gamma'(z)}^{m}= f(z)
\quad \forall z\in \bbH, \; \forall \gamma\in \Gamma.
\label{con.3}\end{equation}
For instance, the natural K\"ahler metric associated to the hyperbolic 
metric $g_{\Sigma}$, seen as a section of $\Lambda_{\Sigma}^{1,0}\otimes 
\Lambda_{\Sigma}^{0,1}$, corresponds to the automorphic form of weight
$(2,2)$
\begin{equation}
        \frac{1}{y^{2}} \;\mbox{on}\;\bbH.
\label{con.4}\end{equation}
In the correspondence between Riemann surfaces and quotients of $\bbH$,
a change of complex structure corresponds to a change of the 
Fuchsian group $\Gamma$.  This provides a canonical identification between
the Teichm\"uller space $T_{g,n}$ of Riemann surfaces of type $(g,n)$ and
the Teichm\"uller space of Fuchsian groups of type $(g,n)$.  Under this 
identification, the tangent space of $T_{g,n}$ at $[\Sigma]$ can
be identified with the subspace 
$\Omega^{-1,1}(\Sigma)= \ker \db_{-1}^{*}\subset \cH_{-1,1}$
of harmonic Beltrami differentials.  Each element
of $\mu\in \Omega^{-1,1}(\Sigma)$ has the form
$\mu= y^{2}\overline{\varphi}$ for a unique 
$\varphi\in \ker \db_{2}$, so that $\dim_{\bbC}\Omega^{-1,1}(\Sigma)=
3g-3+n$.  In particular, an element of $\Omega^{-1,1}(\Sigma)$ decays exponentially
fast as one approaches a puncture (using the coordinates of \eqref{bc.7}).    
The (holomorphic) cotangent space $T^{*}_{[\Sigma]}T_{g,n}$ can be
identified with $\ker\db_{2}$ on $\Sigma$, this space being naturally
dual to $\Omega^{-1,1}(\Sigma)$ via the pairing 
\begin{equation}
(\mu,\varphi):= \int_{\Sigma} \mu\varphi, \quad 
  \mu\in \Omega^{-1,1}(\Sigma), \,\varphi\in \ker\db_{2}.
\label{con.5}\end{equation}
To get complex coordinates on $T_{g,n}$ we can use the fact that to
every $\mu\in \Omega^{-1,1}(\Sigma)$ satisfying
\begin{equation}
        \| \mu\|_{L^{\infty}}= \sup_{z\in\Sigma}|\mu(z)|<1,
\label{con.6}\end{equation}
one can associate a unique diffeomorphism $f^{\mu}:\bbH\to \bbH$
satisfying the Beltrami equation
\begin{equation}
             \frac{\pa f^{\mu}}{\pa \overline{z}}= \mu 
\frac{\pa f^{\mu}}{\pa z}
\label{con.7}\end{equation}
and fixing the points $0,1,\infty$, where $\mu$ in \eqref{con.7} is
seen as an automorphic form of weight $(-2,2)$ on $\bbH$.  From this solution, one gets
a new Fuchsian group by considering $\Gamma^{\mu}:= 
f^{\mu}\Gamma(f^{\mu})^{-1}$,
that is, a new complex structure by considering the Riemann surface
$\Sigma^{\mu}:= \Gamma^{\mu}\setminus \bbH$.
The diffeomorphism $f^{\mu}$ also naturally descends to the quotient
$\Gamma\setminus \bbH$ to give a diffeomorphism
\begin{equation}
    f^{\mu}: \Gamma\setminus\bbH\to \Gamma^{\mu}\setminus\bbH.
\label{con.7b}\end{equation}
Now, if one chooses a basis $\mu_{1}, \ldots, \mu_{3g-3+n}$ of 
$\Omega^{-1,1}(\Sigma)$ and sets $\mu= \varepsilon_{1}\mu_{1}+\cdots
+\varepsilon_{3g-3+n}\mu_{3g-3+n}$,
then the correspondence $(\varepsilon_{1},\ldots,\varepsilon_{3g-3+n})
\mapsto [\Sigma^{\mu}]$ defines complex coordinates in a neighborhood 
of $[\Sigma]\in T_{g,n}$ called \textbf{Bers coordinates}.  In the overlapping of
neighborhoods of two points $[\Sigma]$ and $[\Sigma^{\mu}]$, the Bers
coordinates transform complex analytically (see for instance p.409 in 
\cite{TZ}), defining on $T_{g,n}$ a complex structure.  The Bers coordinates
provide a local trivialization of the fibration 
$p:\cT_{g,n}\to T_{g,n}$ of the universal Teichm\"uller curve, in fact,
of its universal cover, the Bers fibre space $\cB\cF_{g,n}$ (see 
p.138 in \cite{Wolpert}).    
If $\cU\subset T_{g,n}$ is the open set where the Bers coordinates
$(\varepsilon_{1},\ldots,\varepsilon_{3g-g+n})$
associated to $[\Sigma]$ are defined, then this trivialization is 
given by the commutative diagram
\begin{equation}
\xymatrix{  \cU\times \Sigma \ar[r]^{\nu}\ar[rd]^{\pr_{1}} &  
p^{-1}(\cU)\ar[d]^{p} \\
   &  \cU
}
\label{con.8}\end{equation}
where $\pr_{1}$ is the projection on the first factor and $\nu$
is given by $\nu(\mu,\sigma)= f^{\mu}(\sigma)\in p^{-1}([\Sigma^{\mu}])$ where
$f^{\mu}$  denotes the map \eqref{con.7b}.  

This local trivialization also induces a lift of $T_{[\Sigma]}T_{g,n}$ to
$\left.T\cT_{g,n}\right|_{p^{-1}([\Sigma])}$, 
namely (see p.142 in \cite{Wolpert}), a vector $\mu\in T_{[\Sigma]}T_{g,n}$ 
has a canonical lift $pr_{1}^{*}\mu\in \left.T(\cU\times \Sigma)
\right|_{\{[\Sigma]\}\times \Sigma}$, and therefore a canonical lift
$\nu_{*}(\pr_{1}^{*}\mu)\in \left. T\cT_{g,n}\right|_{p^{-1}([\Sigma])}$.
More generally, introducing Bers coordinates at each $[\Sigma]\in T_{g,n}$,
we can get in this way a canonical horizontal lift of $TT_{g,n}$
to $T\cT_{g,n}$.  In other words, associated to the fibration
$p:\cT_{g,n}\to T_{g,n}$, there is a \textbf{canonical connection} $\cP$, that
is, $\cP\subset TT_{g,n}$ is a distribution of hyperplanes 
such that 
\begin{equation}
   p_{*}: \cP_{z}\to T_{p(z)}T_{g,n}
\label{con.10}\end{equation}
is an isomorphism for every $z\in \cT_{g,n}$.  It is also possible to
define a covariant derivative
\begin{equation}
          \nabla^{\cP}: \CI(\cT_{g,n}; (\Lambda^{1,0}_{v})^{\ell}\otimes
                 (\Lambda_{v}^{0,1})^{m})\to \CI(\cT_{g,n};
            p^{*}(T^{*}_{g,n})\otimes (\Lambda^{1,0}_{v})^{\ell}\otimes
                 (\Lambda_{v}^{0,1})^{m}).
\label{con.11}\end{equation}
This allows one to differentiate sections of $(\Lambda_{v}^{1,0})^{\ell}\otimes
(\Lambda_{v}^{0,1})^{m}$ with respect to vectors on the base $T_{g,n}$.
At $[\Sigma]\in T_{g,n}$, the differentiation can be described by
using the Bers coordinates associated to $T_{[\Sigma]}T_{g,n}\cong
\Omega^{-1,1}(\Sigma)$ with the local trivialization \eqref{con.8} of
$p:\cT_{g,n}\to T_{g,n}$ near $[\Sigma]$.  In this trivialization,
a section $\omega$ of $(\Lambda^{1,0}_{v})^{\ell}\otimes
                 (\Lambda_{v}^{0,1})^{m}$ corresponds to a section
$\widetilde{\omega}$ of $(\pr_{2}^{*}\Lambda_{\Sigma}^{1,0})^{\ell}\otimes
(\pr_{2}^{*}\Lambda_{\Sigma}^{0,1})^{m}$ on $\cU\times \Sigma$ where
$\pr_{2}:\cU\times\Sigma\to \Sigma$ is the projection on the second factor.
Precisely, in terms of automorphic forms of weight $(2\ell,2m)$, we have that 
\begin{equation}
 \widetilde{\omega}(\varepsilon,\sigma)= \omega\circ f^{\mu} \left(\frac{\pa f^{\mu}}{\pa z}
 \right)^{\ell} \left(\overline{\frac{\pa f^{\mu}}{\pa z}}
 \right)^{m}
\label{con.12}\end{equation}
where $\mu= \varepsilon_{1}\mu_{1}+\cdots+\varepsilon_{3g-3+n}\mu_{3g-3+n}$.
On $\Sigma= p^{-1}([\Sigma])\subset \cT_{g,n}$, there is a canonical
identification between $(\Lambda^{1,0}_{v})^{\ell}\otimes
                 (\Lambda_{v}^{0,1})^{m}$  and
$(\Lambda^{1,0}_{\Sigma})^{\ell}\otimes
                 (\Lambda_{\Sigma}^{0,1})^{m}$.  Under this identification,
the covariant derivative of $\omega$ takes the form (\cf p.409 in
\cite{TZ}),
\begin{equation}
  \left.\nabla^{\cP}_{\frac{\pa}{\pa \varepsilon_{i}}} \omega
\right|_{p^{-1}([\Sigma])}= \left.
\frac{\pa}{\pa\varepsilon_{i}} \widetilde{\omega}(\varepsilon,\sigma)
\right|_{\varepsilon=0}.
\label{con.13}\end{equation}

An important example is given by the family of fibrewise hyperbolic area forms
$dg_{\Sigma}$, which as was shown in \cite{Ahlfors} gives a parallel section of $\Lambda^{1,1}_{v}$ with respect
to the connection $\cP$,
\[
        \nabla^{\cP} dg_{\Sigma}=0.
\]
This corresponds to the fact that the automorphic form of weight $(2,2)$ $\frac{1}{y^{2}}$ is
parallel with respect to the connection $\cP$.  However, notice that this does not
imply the family of hyperbolic metrics $g_{\Sigma}, [\Sigma]\in T_{g,n}$ is parallel
with respect to $\cP$ as a section of $T^{*}_{v}\cT_{g,n}\otimes T^{*}_{v}\cT_{g,n}$.
In fact, they cannot be parallel with respect to any connection, since otherwise this would mean
that these metrics are all isometric, a contradiction since essentially by definition
of the Teichm\"uller space, these metrics are not even conformal to one another.  

It is also possible to define the covariant derivative of families of operators using the
connection $\cP$.  If $A^{\varepsilon}: \cH_{\ell,m}(\Sigma^{\mu})\to
\cH_{\ell',m'}(\Sigma^{\mu})$ is such family in the trivialization \eqref{con.8} given
by the Bers coordinates, then the covariant derivative of $A^{\varepsilon}$ at $[\Sigma]$ is
given by
\begin{equation}
\begin{gathered}
\left.\nabla^{\cP}_{\frac{\pa}{\pa \varepsilon_{i}}} A^{\varepsilon} \right|_{[\Sigma]} =
  \left. \frac{\pa}{\pa \varepsilon_{i}} (f^{\mu})^{*} A^{\varepsilon}
  (f^{\mu *})^{-1} \right|_{\varepsilon=0},   \\
    \left.\nabla^{\cP}_{\frac{\pa}{\pa \overline{\varepsilon}_{i}}} A^{\varepsilon} \right|_{[\Sigma]} =
  \left. \frac{\pa}{\pa \overline{\varepsilon}_{i}} (f^{\mu})^{*} A^{\varepsilon}
  (f^{\mu *})^{-1} \right|_{\varepsilon=0}. 
\end{gathered}
\label{con.17}\end{equation}
For example, the covariant derivatives of $\db_{\ell}$ and $\db_{\ell}^{*}$ at $[\Sigma]$
are given by (see formula (2.6) in \cite{TZ})
\begin{equation}
\begin{gathered}
   \nabla^{\cP}_{\mu} \db_{\ell}= \mu \db^{*}_{\ell+1}u, \quad \nabla^{\cP}_{\overline{\mu}}\db_{\ell}=0,  \\
   \nabla^{\cP}_{\mu} \db_{\ell}^{*}=0, \quad \nabla^{\cP}_{\overline{\mu}} \db^{*}_{\ell}=
     \overline{\mu}\db_{\ell-1} u^{-1}   
\end{gathered}
\label{con.18}\end{equation}
where $u:=\frac{1}{y^{2}}$ is seen as a section of $\Lambda^{1,0}_{\Sigma}\otimes \Lambda^{0,1}_{\Sigma}$.

As we have seen, each Riemann surface $\Sigma$ of type $(g,n)$ has a boundary compactification
$\Sigma_{\hc}$ constructed using the metric $g_{\Sigma}$.  These compactifications fit together
to give a fibrewise boundary compactification ${}^{\hc}\cT_{g,n}$ of the universal
Teichm\"uller curve.   In terms of the local trivializations of \eqref{con.8}, this is 
because the solution $f^{\mu}$ to the Beltrami equation \eqref{con.7} is
real analytic (see for instance proposition 4.6.2 in \cite{Hubbard}), it maps the fixed
points of $\Gamma$ to the fixed points of $\Gamma^{\mu}$ and, seen as a map 
$f^{\mu}:\Sigma\to \Sigma^{\mu}$, it is asymptotically holomorphic as one approaches any puncture of $\Sigma$.  Since the canonical connection $\cP$ is obtained by using Bers coordinates and infinitesimal deformations induced by the solutions of the Beltrami equation
\eqref{con.7}, we see that it also naturally lifts to provide a canonical connection 
${}^{\hc}\cP$ to the fibration 
\[
{}^{\hc}p: {}^{\hc}\cT_{g,n}\to T_{g,n}.
\]

To get a natural boundary defining function for ${}^{\hc}\cT_{g,n}$, we use the construction
of \eqref{bc.10} in each fibre.  This definition depends on the choice of a number
$\epsilon>0$ which has to be chosen so that each cusp end $N_{i}$ in a given 
surface has area strictly
greater than $\epsilon$.  To get a global definition ${}^{\hc}\cT_{g,n}$, we should replace 
the number $\epsilon$ by a smooth function $a: T_{g,n}\to \bbR^{+}$ such that in a given
fibre $\Sigma:=p^{-1}([\Sigma])$, the area of each cusp end $N_{i}$ is strictly greater than
$a([\Sigma])$.  We can then define our global defining function on ${}^{\hc}\cT_{g,n}$ to
be 
\begin{equation}
   \rho(\sigma)= \rho_{\Sigma,a([\Sigma])}(\sigma) \quad \mbox \quad 
   \mbox{for} \quad 
\sigma\in \Sigma:=p^{-1}([\Sigma]), \quad [\Sigma]\in T_{g,n}
\label{con.19}\end{equation}
where $\rho_{\Sigma, \epsilon}:\Sigma_{\hc}\to \bbR$ is defined in \eqref{bc.10} for the
Riemann surface $\Sigma$ and a choice of small $\epsilon>0$.

\section{A local formula for the family index}
\label{lf.0}

The family of operators $\db_{\ell}\in \Psi^{1}(\cT_{g,n}/T_{g,n};
{}^{\hc}K^{\ell}_{v}, {}^{\hc}\Lambda^{0,1}_{v}\otimes {}^{\hc}K^{\ell}_{v})$ 
is a particular example of the families of $\fD$ operators considered in 
\cite{Albin-Rochon}.  When we apply this local index theorem to
our family $\db_{\ell}$ with the canonical connection ${}^{\hc}\cP$ for
the fibration ${}^{\hc}p: {}^{\hc}\cT_{g,n}\to T_{g,n}$, we get the family version
of \eqref{rr.8}, 
\begin{multline}
\Ch( \Ind(\db_{\ell}))= 
      \int_{\cT_{g,n}/T_{g,n}}\Ch(T^{-\ell}_{v}(\cT_{g,n}))\Td(T_{v}\cT_{g,n})
-\hat{\eta}(D^{V}_{\ell}) \\
 -\hat{\eta}(D^{H}_{\ell}) -\lrpar{\frac{1}{2\pi \sqrt{-1}}}^{\frac{N}{2}}d\int_{0}^{\infty} \Str\lrpar{ \frac{\pa \bbA^{t}_{D_{\ell}}}{\pa t} e^{-(\bbA^{t}_{D_{\ell}})^{2}}} dt,
\label{lf.1}\end{multline}
where the eta invariants of the vertical and horizontal families are replaced by the
the corresponding eta forms of Bismut and Cheeger \cite{BCO} (with non-standard $\bbZ_2$ grading for $D^H$).
This is an equality at the level of forms.  Notice that in \cite{Albin-Rochon} the first term is expressed in terms of the $\hat{A}$ form.  However,
thanks to Theorem 5.5 in \cite{Wolpert} and its reformulation in
equation 5.3 of \cite{Wolpert}, the fibration
$p:\cT_{g,n}\to T_{g,n}$ is K\"ahler fibration (see \cite{BGS2} for a defintion) so that it is possible to rewrite the first term using the Todd form instead.  In the last term, $\bbA^{t}_{D_{\ell}}$ is the
rescaled Bismut superconnection while $N$ is the number operator in $\Lambda T_{g,n}$, that is, the action of $N$ on forms of degree $k$ on $T_{g,n}$ is multiplication by $k$.  
\begin{remark}
In this paper, our convention for the Chern character differs from that of \cite{BGV}.
This is why we need to include these extra factors of $2\pi i$ in
the last term.  In principle, the eta forms would also require such factors, so really, by 
an eta form, we mean $(2\pi i)^{-\frac{N}{2}}$ times the eta form of Bismut and Cheeger
(\cf equation 4.101 in \cite{BCO}).  
\label{convention.1}\end{remark}
When we take the degree
zero part of \eqref{lf.1}, we get back the numerical index
\eqref{rr.8} by evaluating it at a given point $[\Sigma]\in T_{g,n}$.
In fact, as we have seen, the degree zero part of $\widehat{\eta}(D_{\ell}^{V})$ is identically zero, while 
the degree zero part of $\widehat{\eta}(D^{H}_{\ell})$ is $\frac{n}{2} \sign(\ell-\frac{1}{2})$.
However, the higher degree components of $\hat{\eta}(D^{H}_{\ell})$ 
vanish identically at the level of forms as we will see in a moment.

Let
$\pa\cT_{g,n}:=\rho^{-1}(0)$ be the union of the boundaries of the
fibres of ${}^{\hc}p: {}^{\hc}\cT_{g,n}\to T_{g,n}$.  The map ${}^{\hc}p$
induces a fibration structure
\begin{equation}
   \pa p: \pa \cT_{g,n}\to T_{g,n}
\label{cef.1}\end{equation}  
with typical fibre the disjoint union of $n$ circles.  In fact, the manifold
$\pa \cT_{g,n}$ has precisely $n$ components,
\begin{equation}
   \pa\cT_{g,n}= \bigcup_{i=1}^{n}\pa_{i}\cT_{g,n}
\label{cef.2b}\end{equation}
with $\pa_{i}\cT_{g,n}$ the component associated to the $i$th cusp.
There is a corresponding fibration structure
   \begin{equation}
\pa p_{i}: \pa_{i} \cT_{g,n}\to T_{g,n}.
\label{cef.2}\end{equation}
Recall that the vertical family $D^{V}_{\ell}$ decomposes as 
\begin{equation}
  D^{V}_{\ell} = \lrpar{ \begin{array}{cc}
                          0 & D^{V,-}_{\ell}  \\
                      D^{V,+}_{\ell} & 0
                       \end{array}}
\label{z2.1}\end{equation}
with respect to the $\bbZ_{2}$ grading of the Clifford bundle.

\begin{lemma}
The Chern form of $\ker D^{V,+}_{\ell} \to T_{g,n}$ vanishes in positive degrees,
\[
         \Ch(\ker D^{V,+}_{\ell})_{[2k]}=0, \quad k\in \bbN.
\]
\label{lf.108}\end{lemma}
\begin{proof}
Let $D_{\ell,i}^{V}$ be the vertical family of the $i$th component
$\pa_{i} \cT_{g,n}$ of  $\pa\cT_{g,n}$.
Via the identification
\[
-c(\frac{d\rho}{\rho}): \,{}^{\hc}\Lambda_{v}^{0,1}\otimes{}^{\hc}K_{v}
 \longrightarrow{}^{\hc}K_{v} 
\] 
given by Clifford multiplication,
the operator $D^{V,+}_{\ell}$ can be identified with
\begin{equation}
    \frac{1}{i} \nabla_{\frac{\pa}{\pa x}}= i\nabla_{\frac{\pa}{\pa u}}:  
  \CI(\bbR/\bbZ; {}^{\hc}K^{\ell}_{v})\to \CI(\bbR/\bbZ; {}^{\hc}K^{\ell}_{v})
\label{rr.9b}\end{equation}
where $u=-x$ is such that $\frac{\pa}{\pa u}$ is an oriented orthonormal basis
of $T_{\sigma}(\pa_{i}\cT_{g,n}/T_{g,n})$ for each $\sigma\in \pa_{i}\cT_{g,n}$.
Thus, 
$\ker D_{\ell,i}^{V,+}\to T_{g,n}$ defines a complex line bundle over
the Teichm\"uller space $T_{g,n}$ and 
\begin{equation}
  \ker D_{\ell}^{V,+}= \bigoplus_{i=1}^{n} \ker D^{V,+}_{\ell,i}.
\label{lf.4}\end{equation}
For the corresponding Chern characters, this gives
\begin{equation}
  \Ch(\ker D_{\ell}^{V,+})= \sum_{i=1}^{n} \Ch(\ker D_{\ell,i}^{V,+}).
\label{lf.5}\end{equation}
To prove the lemma, it therefore suffices to show that 
$\Ch(\ker D^{V,+}_{\ell,i})_{[2]}=0$ for $i\in\{1,\ldots,n\}$.  This will
be true provided we can trivialize $\ker D^{V,+}_{\ell,i}$ by a parallel 
section.  From the identification of $D^{V,+}_{\ell,i}$ with
\eqref{rr.9b}, a choice of trivializing section is given by taking
\begin{equation}
s_{\ell,i}:[\Sigma]\mapsto \left.\left(\rho dx -i \frac{d\rho}{\rho}\right)^{\ell}
\right|_{(\pa \Sigma_{\hc})_{i}} \in \left.{}^{\hc}K_{v}^{\ell}
\right|_{(\pa \Sigma_{\hc})_{i}}
\label{lf.6}\end{equation} 
where $\Sigma= p^{-1}([\Sigma])$ and $\rho$ is the boundary defining 
function of \eqref{con.19}.  Notice that the section \eqref{lf.6}
is completely determined by the canonical family of hyperbolic
metrics $g_{\cT_{g,n}/T_{g,n}}$.  Conversely, for $\ell=1$,
the section $s_{1,i}$ completely determines the asymptotic behavior
of $g_{\cT_{g,n}/T_{g,n}}$ as one approaches the $i^{\text{th}}$ puncture.
If the family of metrics $g_{\cT_{g,n}/T_{g,n}}$ were parallel with 
respect to the canonical connection $\cP$, we could conclude immediately that
the section $s_{\ell,i}$ is parallel.  This is not the case, but at
least the family of metrics $g_{\cT_{g,n}/T_{g,n}}$ is asymptotically 
parallel as one approaches a puncture.  Indeed, from \eqref{con.18}, we
see that the parallel transport (along a path on $T_{g,n}$)  
defined by the canonical connection $\cP$ is asymptotically holomorphic
as one approaches a puncture.  This is because the Beltrami 
differential $\mu$ in \eqref{con.7} vanishes exponentially fast as one
approaches a puncture (using the coordinates of \eqref{bc.7}).  Thus,
parallel transport is asymptotically a conformal transformation for the
family of metrics  $g_{\cT_{g,n}/T_{g,n}}$.  Since
\begin{equation}
            \nabla^{\cP} dg_{\cT_{g,n}/T_{g,n}}=0,
\label{lf.7}\end{equation}
this means that the parallel transport defined by the connection $\cP$ is
asymptotically an isometry as one approaches a puncture.  That is,
$\nabla^{\cP} g_{\cT_{g,n}/T_{g,n}}$ is asymptotically zero as one 
approaches a puncture.  In particular, this implies that for each
$i\in \{1,\ldots,n\}$, the section $s_{\ell,i}$ of \eqref{lf.6} is 
parallel with respect to the connection $\cP$.

\end{proof}

Together with the boundary defining function $\rho$, the family
of metric $g_{\cT_{g,n}/T_{g,n}}$ induces a natural family of 
metrics  $g_{i}$ for each fibre of the fibration \eqref{cef.2}
in such
a way that each fibre becomes isometric to the circle $\bbS^{1}:=\bbR/\bbZ$ of
length $1$ (\cf \cite{Wolpert2}).  With these identifications, we get a natural action of 
$\bbS^{1}$ on each fibre, giving \eqref{cef.2} the structure of a 
principal $\bbS^{1}$-bundle.  By construction, the family of metrics $g_{i}$
is $\bbS^{1}$-equivariant with respect to the $\bbS^{1}$ action.
The canonical connection ${}^{\hc}\cP$ naturally
induces a connection $\cP_{i}$ on \eqref{cef.2}.
\begin{lemma}
The family of metrics $g_{i}$ is parallel with respect to the connection
$\cP_{i}$, that is, the connection $\cP_{i}$ is unitary with respect to the
metric $g_{i}$.  In particular, on the $i$th circular boundary component, the
vector field $\frac{\pa}{\pa u}$ is parallel with respect to the connection $\mathcal{P}_{i}$.  
\label{cef.3}\end{lemma}
\begin{proof}
By the proof of lemma~\ref{lf.108}, the family of metrics 
$g_{\cT_{g,n}/T_{g,n}}$ 
is asymptotically parallel as one approaches a cusp, from which the result follows.
\end{proof}
We can now show that the eta form of $D^{H}_{\ell}$ vanishes in positive degrees.

\begin{lemma}
For each $k\in \bbN$, the degree $2k$ part of the form
$\widehat{\eta}(D^{H}_{\ell})$ vanishes identically,
\[
            \widehat{\eta}(D^{H}_{\ell})_{[2k]}=0, \quad 
k>0.
\]
\label{lf.3}\end{lemma}
\begin{proof}
Since $D^{H}_{\ell}$ is just an endomorphism of $\ker D^{V}_{\ell}$, we see from proposition~\ref{fp.2}, lemma~\ref{cef.3} and the definition of the eta form that (see \cite[(4.12)]{Albin-Rochon})
\begin{equation}
   \widehat{\eta}(D^{H}_{\ell})= \frac{1}{2}\sign\lrpar{\ell-\frac{1}{2}} \Ch(\ker D^{V,+}_{\ell}).
\label{etah}\end{equation} 
The result then follows from lemma~\ref{lf.108}.  
\end{proof}

On the other hand, the eta form of the vertical family gives a  contribution in higher degrees.  In fact, since the geometry of the boundary fibration is very special, it is
possible to compute the eta form explicitly.  
With respect to the decomposition \eqref{cef.2b}, the vertical family 
$D_{\ell}^{V}$ admits a corresponding decomposition
\begin{equation}
   D_{\ell}^{V}= \bigoplus_{i=1}^{n} D_{\ell,i}^{V}
\label{cef.4}\end{equation}
where $D^{V}_{\ell,i}$ is a family of self-adjoint Dirac operators on
the fibration \eqref{cef.2}.  In terms of this decomposition, the eta form
of $D^{V}_{\ell}$ can be expressed  as
\begin{equation}
   \widehat{\eta}(D^{V}_{\ell})= \sum_{i=1}^{n} \widehat{\eta}(D^{V}_{\ell,i}).
\label{cef.5}\end{equation}
By \eqref{con.18} (see also the proof of lemma~\ref{qc.10}), the family of
Dirac operators $D_{\ell}$ is asymptotically parallel with respect to the
canonical connection $\cP$ as one approaches a cusp.  This means that each
of the vertical families $D^{V}_{\ell,i}$ is parallel with respect to the 
connection $\cP_{i}$ on \eqref{cef.2}.  This fact, together with the fact
the family of metric $g_{i}$ is parallel with respect to the connection
$\cP_{i}$ and is equivariant with respect to the circle action,  means that
we can 
apply the result of Zhang (Theorem 1.7 in \cite{Zhang}) to get an explicit formula for the eta form
$\eta(D^{V}_{\ell,i})$.   
\begin{proposition}[Zhang, \cite{Zhang}, Theorem 1.7]
The eta form of $D^{V}_{\ell,i}$ is given by 
\[
     \widehat{\eta}(D^{V}_{\ell,i})= \frac{1}{2\tanh \lrpar{\frac{e_{i}}{2}}}
-\frac{1}{e_{i}}
\]
where $e_{i}:= \frac{\sqrt{-1}}{2\pi}\Theta_{i}$ is the curvature form of the circle bundle 
$\pa p_{i}: \pa_{i}\cT_{g,n}\to T_{g,n}$
with connection $\cP_{i}$ and curvature $\Theta_{i}$, the Lie algebra of $\bbS^{1}$ being
identified with $i\bbR$. 
\label{cef.6}\end{proposition}
\begin{remark}
Notice in particular that this implies that the eta form is zero in degree $2k$ for 
$k=0 $ modulo 2.  Moreover, it is a closed form, an unusual feature for a
eta form.  
\label{cef.11}\end{remark}

Before stating our main theorem, let us give an alternate description
of the Chern form $e_{i}$.  Namely, to the circle bundle
\eqref{cef.2} with connection $\cP_{i}$ and family of metrics $(2\pi)g_{i}$,
we can associate in a canonical way a complex line bundle $\cL_{i}\to T_{g,n}$
equipped with a Hermitian metric $h_{i}$ and a unitary connection
$\nabla^{\cL_{i}}$ in such a way that the curvature form of $\cL_{i}$ is
precisely $(-2\pi\sqrt{-1})e_{i}$.  The line bundle $\cL_{i}$ is such that
its unit circle bundle with induced metric and connection is precisely
the circle bundle \eqref{cef.2} with family of metrics $2\pi g_{i}$
and connection $\cP_{i}$.

Thinking of a fibre $\Sigma:=p^{-1}([\Sigma])$ as a punctured Riemann surface
\begin{equation}
  \Sigma= \overline{\Sigma}-\{x_{1},\ldots,x_{n}\},
\label{cef.7}\end{equation}  
one can also define the line bundle $\cL_{i}$ by
\begin{equation}
   \cL_{i,[\Sigma]}:= (T^{1,0}_{x_{i}}\overline{\Sigma})^{*}= 
\left.K_{\overline{\Sigma}} \right|_{x_{i}}, \quad [\Sigma]\in T_{g,n}.
\label{cef.8}\end{equation}
Moreover, from this perspective, the Hermitian metric $h_{i}$ and the 
unitary connection $\nabla^{\cL_{i}}$ are easily seen to be the same as the
one introduced by Wolpert \cite{Wolpert2}.  Thus, the form $e_{i}$ corresponds
to the Chern form $c_{1}(\|\; \|_{can,i})$ of Corollary 7 in \cite{Wolpert2}.

Now, combining \eqref{lf.1} with Lemma~\ref{lf.3} and Proposition~\ref{cef.6},
we obtain the following formula.
\begin{theorem}
The local family index of the family of operators 
\[
D_{\ell}^{+}:=\sqrt{2} \ \db_{\ell}\in \rho^{-1}\Psi^{1}_{\cusp}(\cT_{g,n}/T_{g,n}; 
       {}^{\hc}K^{\ell}_{v}, {}^{\hc}\Lambda^{0,1}_{v}\otimes 
           {}^{\hc}K^{\ell}_{v})
\]
associated to the Teichm\"uller universal curve 
$p:\cT_{g,n}\to T_{g,n}$ and its canonical connection $\cP$ is given by
\begin{multline}
\Ch(\ind\ker D^{+}_{\ell})= 
      \int_{\cT_{g,n}/T_{g,n}}\Ch(T^{-\ell}_{v}(\cT_{g,n}))\Td(T_{v}\cT_{g,n})
+\frac{n}{2}\sign \lrpar{\frac{1}{2}-\ell} \\
-\sum_{i=1}^{n}
\lrpar{ \frac{1}{2\tanh \lrpar{\frac{e_{i}}{2}}}-\frac{1}{e_{i}}}  -
\lrpar{\frac{1}{2\pi\sqrt{-1}}}^{\frac{N}{2}}
d\int_{0}^{\infty} \Str\lrpar{ \frac{\pa \bbA^{t}_{D_{\ell}}}{\pa t} e^{-(\bbA^{t}_{D_{\ell}})^{2}}} dt,
\label{lf.10b}\end{multline}
where $\bbA^{t}_{D_{\ell}}$ is the rescaled Bismut superconnection associated to the family 
$D_{\ell}$, $e_{i}$ is the canonical Chern form of the 
(holomorphic) cotangent bundle along the $i^{\text{th}}$ cusp $\cL_{i}\to T_{g,n}$ and
$N$ is the number operator on $\Lambda T_{g,n}$.
\label{lf.8}\end{theorem}

As in \cite{TZ}, each of the terms in our formula is invariant under the action of the
Teichm\"uller modular group $\Mod_{g,n}$.  Thus, formula \eqref{lf.10b} also
holds on the moduli space $\mathcal{M}_{g,n}:= T_{g,n}/\Mod_{g,n}$ in the 
sense of orbifolds with the fibration $p:\cT_{g,n}\to T_{g,n}$ replaced by the
forgetful map $\pi_{n+1}:\cM_{g,n+1}\to \cM_{g,n}$.  In fact, on the moduli space
$\mathcal{M}_{g,n}$, the formula acquires a topological meaning in higher degrees.

To see this, define $\overline{\cT}_{g,n}$ to be the space obtained from $\cT_{g,n}$ by filling each
puncture of each fibre by a marked point.  There is still a fibration
$\overline{p}: \overline{\cT}_{g,n}\to T_{g,n}$, but now with fibres being compact 
Riemann surfaces of genus $g$ with $n$ marked points.  Let $\overline{K}_{v}\to \cT_{g,n}$
denote the corresponding vertical canonical line bundle (the dual of the vertical $(1,0)$ tangent
bundle).  Let $D_{i}\subset \overline{\cT}_{g,n}$ be the divisor associated to the $i$th marked
points and let $L_{D}$ be the line bundle associated to the divisor $D:= \sum_{i=1}^{n}D_{i}$.  Then, by analogy with the discussion in \S~\ref{rr.0}, we see that the family index of 
$\db_{\ell}$ is the same as the family index of the family of $\db$-operators
\begin{equation}
  \widehat{\db}_{\ell}: \CI(\cT_{g,n};\overline{K}_{v}^{\ell}\otimes L_{D}^{\ell-1})\to 
    \CI(\cT_{g,n};\Lambda^{0,1}_{v}\otimes\overline{K}_{v}^{\ell}\otimes L_{D}^{\ell-1})
\label{grr.1}\end{equation}
for $\ell>0$ and 
\begin{equation}
  \widehat{\db}_{\ell}: \CI(\cT_{g,n};\overline{K}_{v}^{\ell}\otimes L_{D}^{\ell})\to 
    \CI(\cT_{g,n};\Lambda^{0,1}_{v}\otimes\overline{K}_{v}^{\ell}\otimes L_{D}^{\ell})
\label{grr.2}\end{equation}
for $\ell\le 0$.  On the fibration $\pi_{n+1}: \cM_{g,n+1}\to \cM_{g,n}$, this corresponds to
the following situation.  Let $\omega_{\pi_{n+1}}$ be the relative dualizing sheaf of this 
fibration, that is, the sheaf of sections of $\overline{K}_{v}$.  Let $\omega_{\pi_{n+1}}(D)$ be
the logarithmic variant of $\omega_{\pi_{n+1}}$, which means that the local sections of
$\omega_{\pi_{n+1}}(D)$ are sections of $\omega_{\pi_{n+1}}$ with possibly simple poles at the
first $n$ marked points.  Then the line bundle $\overline{K}_{v}\otimes L_{D}$ on $\cT_{g,n}$ 
corresponds to the sheaf $\omega_{\pi_{n+1}}(D)$ on $\cM_{g,n+1}$.

Going back to the formula of theorem~\ref{lf.8}, we see that 
the form $e_{i}$ then represents the Miller
class $\psi_{i}= c_{1}(\cL_{i})$.  On the other hand, since the Miller class $\psi_{n+1}$ on $\cM_{g,n+1}$ is given by $\psi_{n+1}= c_{1}(\omega_{\pi_{n+1}}(D))$ (see for instance p.254 in \cite{Witten}), the first term in the right-hand side of \eqref{lf.10b} can be seen to represent
a linear combination of the Mumford-Morita classes
\begin{equation}
    \kappa_{j}:= (\pi_{n+1})_{*}( \psi^{j+1}_{n+1}) =\left[(\pi_{n+1})_{*}(e_{n+1}^{j+1})\right], \quad j\in\bbN_{0},
\label{grr.4}\end{equation}
where $e_{n+1}$ is the Chern form of the vertical canonical line bundle $K_{v}\cong
\overline{K}_{v}\otimes L_{D}$.
The precise formula involves the Bernouilli numbers $B_{m}$ and the 
Bernouilli polynomials $B_{m}(\ell)$, which are defined by the following identities,
\begin{equation}
\frac{x}{e^{x}-1}= \sum_{m\ge 0} B_{m} \frac{x^{m}}{m!}, \quad
\frac{e^{\ell x}x}{e^{x}-1}= \sum_{m\ge 0} B_{m}(\ell) \frac{x^{m}}{m!}.
\label{Bernouilli}\end{equation}
Thus, the first term in \eqref{lf.10b} is seen to represent the cohomology class
\begin{equation}
 (\pi_{n+1})_{*} \lrpar{ \frac{e^{\ell \psi_{n+1}}\psi_{n+1}}{e^{\psi_{n+1}}-1}}=
   \sum_{m\ge 1} B_{m}(\ell) \frac{\kappa_{m-1}}{m!}.
\label{tc.1}\end{equation}
On the moduli space, theorem~\ref{lf.8} therefore gives the following local formula
(in the sense of orbifolds).
\begin{corollary}\label{cor:Moduli}
In the sense of orbifolds, the Chern character of the index of the family $\db_{\ell}$ 
associated to the forgetful map $\pi_{n+1}\cM_{g,n+1}\to \cM_{g,n}$ is given at the form level
by
\begin{multline}
\Ch(\ker D^{+}_{\ell})-\Ch (\ker D^{-}_{\ell})= 
      \sum_{m\ge 1} \frac{B_{m}(\ell)}{m!} k_{m-1} 
+\frac{n}{2}\sign \lrpar{\frac{1}{2}-\ell} \\
-\sum_{i=1}^{n}
\lrpar{ \frac{1}{2\tanh \lrpar{\frac{e_{i}}{2}}}-\frac{1}{e_{i}}}  
-\lrpar{\frac{1}{2\pi \sqrt{-1}}}^{\frac{N}{2}}d\int_{0}^{\infty} \Str\lrpar{ \frac{\pa \bbA^{t}_{D_{\ell}}}{\pa t} e^{-(\bbA^{t}_{D_{\ell}})^{2}}} dt,
\label{fm.121}\end{multline}
where $k_{m}= (\pi_{n+1})_{*}(e_{n+1}^{m+1})$ and $e_{i}$ are canonical form
representatives of the Morita-Mumford-Miller classes $\kappa_{m}$ and $\psi_{i}$.
\label{fm.122}\end{corollary}

If $\overline{\cM}_{g,n}$ denote the Deligne-Mumford compactification of the moduli space
$\cM_{g,n}$, then theorem~\ref{lf.8} can be intuitively interpreted as a local version of the
Grothendieck-Riemann-Roch theorem applied to the morphism $\pi_{n+1}:\overline{\cM}_{g,n+1}\to
\overline{\cM}_{g,n}$ and the sheaf
\begin{equation}
      \widetilde{\omega}_{\ell}:= \left\{ \begin{array}{ll}
          \omega_{\pi_{n+1}}(D)^{\ell-1}\otimes \omega_{\pi_{n+1}}, & \ell > 0,  \\
           \omega_{\pi_{n+1}}(D)^{\ell}, & \ell \le 0.
      \end{array}
      \right.
\label{grr.3}\end{equation}
In this context, the Grothendieck-Riemann-Roch theorem was first studied and used by Mumford
\cite{Mumford} in the case $n=0$ with formula given by
\[
   \Ch( (\pi_{1})_{*}\omega_{\pi_{1}}^{\ell})= 
   \sum_{m\ge 1} \frac{B_{m}(\ell)}{m!} \kappa_{m-1}+ (\mbox{terms coming from
   } \pa\overline{\cM}_{g}).
\]
When $n>0$, a Grothendieck-Riemann-Roch formula was obtained for the sheaf 
$\omega_{\pi_{n+1}}^{\ell}$ by Bini \cite{Bini},
\begin{equation}
\Ch( (\pi_{n+1})_{*}\omega_{\pi_{n+1}}^{\ell})= 
   \sum_{m\ge 1} \frac{B_{m}(\ell)}{m!} \widetilde{\kappa}_{m-1}+ (\mbox{terms coming from
   } \pa\overline{\cM}_{g,n}).
\label{bini}\end{equation}
where $\widetilde{\kappa}_{m}:= (\pi_{n+1})_{*}\lrpar{ c_{1}( \omega_{\pi_{n+1}})^{m+1}}$.
When $\ell=0$ and $\ell=1$, it makes sense to compare our formula with the one of
Bini.  In that case, using the relation
\begin{equation*}
  \kappa_{m}= \widetilde{\kappa}_{m} + \sum_{i=1}^{n} \psi_{i}^{m}
\label{AC}\end{equation*}
proved by Arbarello and Cornalba \cite{AC} together with the identity
\[
     \frac{x}{2 \tanh{ \frac{x}{2}}}= \frac{x}{e^{x}-1}+ \frac{x}{2},
\]
we can easily check that, as expected, our formula agrees with the interior contribution of \eqref{bini}.

\section{The spectral $\hc$-zeta determinant }\label{Selberg.0} $ $\newline
On any geometrically finite hyperbolic surface $\Sigma=\Gamma\setminus\bbH$, the Selberg's zeta function is defined for $\Re\lrpar s>1$ to be
\begin{equation}
	Z_\Sigma\lrpar s = \prod_{\{\gamma\}}\prod_{k=0}^\infty\lrpar{1-e^{-\lrpar{s+k}\ell\lrpar\gamma}}
\label{Selberg.1}\end{equation}
where the outer product goes over conjugacy classes of primitive {\em hyperbolic} elements of $\Gamma$ and $\ell\lrpar\gamma$ is the length of the corresponding closed geodesic.

On closed hyperbolic surfaces, a well-known result of D'Hoker and Phong 
\cite{DP} says that the determinant of the Laplacian $\Delta_{\Sigma, \ell}$ acting
on sections of $K^{\ell}$ can be expressed in terms of special values
of the Selberg's Zeta function,
\begin{equation}
  \begin{array}{ll}
  \det(\Delta_{\Sigma, \ell})= Z_{\Sigma}(\ell) e^{-c_{\ell-1}\chi(\Sigma)},&
    \ell\ge 2, \\
   \det'(\Delta_{\Sigma,\ell})= Z_{\Sigma}'(1)e^{-c_{0}\chi(\Sigma)}, &
    \ell=0,1.
  \end{array}
\label{DP.1}\end{equation}
where 
\begin{multline}
  c_{\ell}:= \sum_{0\le m< \ell -\frac{1}{2}} ( 2\ell-2m-1)\log(2\ell-m)
  -\left(\ell+\frac{1}{2}\right)^{2} \\
 + \left( \ell +\frac{1}{2}\right) \log{2\pi} + 2\zeta_{Riem}(-1).
\label{DP.2}\end{multline}
Shortly after, it was shown by Sarnak \cite{Sarnak} that for the geometric Laplacian with non-negative spectrum $\Delta_{\Sigma}$,
\begin{equation}\label{Sarnak}
	\frac{\det\lrpar{\Delta_{\Sigma}+s\lrpar{s-1}}} {Z_{\Sigma}\lrpar s} =
	\lrpar{e^{E-s\lrpar{s-1}}\frac{\Gamma_2\lrpar s^2}{\Gamma\lrpar s}\lrpar{2\pi}^s}^{-\chi\lrpar \Sigma}
\end{equation}
where $E=-\frac14-\frac12\log2\pi+2\zeta_{Riem}'\lrpar{-1}$, $\Gamma_2$ is the Barnes double Gamma function.  As indicated in \cite{Sarnak}, the
formula of D'Hoker and Phong can be recovered relatively easily from
\eqref{Sarnak}.  

On a Riemann surface with cusps, the Selberg Zeta function as defined above still makes sense.  
However, since the Laplacian has a continuous spectrum, the definition of its determinant
is more subtle.  It was studied by Efrat \cite{Efrat1}, \cite{Efrat2} and by M\"uller \cite{Muller} using scattering theory to understand the contribution from the continuous spectrum.  
In this paper, we use renormalized integrals to extend the usual definition of the determinant via zeta-regularization to these manifolds, with the advantage that this does not require the metric to have constant curvature. We then use the analysis of \cite{BJP} to show that, on hyperbolic surfaces, our definition satisfies \eqref{Sarnak} with the right-hand-side replaced by a meromorphic function depending only on the genus and the number of punctures, an important feature for our purposes.

\subsection{The determinant of the Laplacian} \label{detL.0}  $ $ \newline

To relate the determinant with the Selberg Zeta function and get an analog of
formula \eqref{Sarnak}, it is convenient to work first with the (positive) geometric Laplacian $\Delta_{\Sigma}$ instead of the $\overline{\pa}$-Laplacian.  Recall that the two are the same modulo a multiplicative constant,
\[
             \Delta_{\overline{\pa}}= \frac{1}{2} \Delta_{\Sigma}.
\]

Following \cite[$\S$9.5]{APSbook} and \cite[$\S$3]{Hassell}, we define the zeta function of $\Delta_{\Sigma}$ using the renormalized trace (see e.g., \cite{Albin-Rochon})
\begin{equation*}
	\zeta_{\Delta_{\Sigma}}(z) := \frac1{\Gamma(z)}
	\int_0^{\infty} t^z \; {}^R \Tr\lrpar{ e^{-t\Delta_\Sigma}-\cP_{\ker\Delta_\Sigma} } \; \frac{dt}t.
\end{equation*}
Since $\Delta_\Sigma$ is Fredholm, zero is spectrally isolated and the integrand decays exponentially for large times. Thus the integral defines a holomorphic function for $\Re(z)$ large enough. The small-times asymptotics of the integrand (whose existence follows from the construction of the heat kernel in \cite{Vaillant}) allow us to extend the function meromorphically to the whole complex plane. We denote the meromorphic extension by the same symbol and define
\begin{equation*}
	\log\det\Delta_{\Sigma} := -\zeta_{\Delta_{\Sigma}}'\lrpar 0.
\end{equation*}

We can find a more explicit expression for the zeta function by subtracting the first few terms in the expansion of the heat kernel at $t=0$. The form of this expansion can be deduced for arbitrary $\fD$ operators of Laplace-type from Vaillant's construction (see the appendix of \cite{Albin-Rochon3} for such an approach), but for the case at hand the expansion is well-known (see, e.g., (2.3) in \cite{Muller})
\begin{equation}
  {}^{R}\Tr( e^{-t\Delta_{\Sigma}}) = \frac{a_{-1}}{t} + \ta_{-\frac{1}{2}}\frac{\log t}{\sqrt{t}}
    + \frac{a_{-\frac{1}{2}}}{\sqrt{t}} + a_{0} + \mathcal{O}(\sqrt{t}) \quad \mbox{as} \; t\to 0^{+}.
\label{expansion.1}\end{equation}
Thus, writing $f_0\lrpar{t}= a_{-1}t^{-1}+ \ta_{-\frac12}\frac{\log t}{\sqrt{t}}+a_{-\frac12}t^{-\frac12}+a_0$ and choosing any $C>0$, we have the expression
\begin{equation*}\begin{array}{ll}
	\zeta_{\Delta_{\Sigma}}\lrpar z &=  
	\frac1{\Gamma\lrpar z}
	\int_0^C t^z \lrpar{{}^R\Tr\lrpar{e^{-t\Delta_{\Sigma}}}-f_0\lrpar t} \;\frac{dt}t\\
	&+\frac1{\Gamma\lrpar z}
	\int_C^\infty t^z \lrpar{{}^R\Tr\lrpar{e^{-t\Delta_{\Sigma}}}-\dim\ker_{-}\lrpar{\Delta_{\Sigma}}} \;\frac{dt}t\\
	&+\frac{C^z}{\Gamma\lrpar{z+1}}\lrpar{a_0-\dim\ker_{-}\lrpar{\Delta_{\Sigma}}}
	+\frac{C^{z-\frac12}a_{-\frac12}}{\lrpar{z-\frac12}\Gamma\lrpar z}  \\
	&+ \frac{\ta_{-\frac12}C^{z-\frac12}}{\Gamma(z)}\left( \frac{\log C}{z-\frac12}- \frac{1}{(z-\frac12)^{2}}
     \right)	
	+\frac{C^{z-1}a_{-1}}{\lrpar{z-1}\Gamma\lrpar z}.
\end{array}\end{equation*}
Differentiating and setting $z=0$, we get 
\begin{equation}\begin{split}
	\zeta_{\Delta_{\Sigma}}'\lrpar 0 &=  
	\int_0^C \lrpar{{}^R\Tr\lrpar{e^{-t\Delta_{\Sigma}}}-f_0\lrpar t} \;\frac{dt}t\\
	&+
	\int_C^\infty \lrpar{{}^R\Tr\lrpar{e^{-t\Delta_{\Sigma}}}-\dim\ker_{-}\lrpar{\Delta_{\Sigma}}} \;\frac{dt}t\\
	&+\lrpar{\log C + \gamma_e}\lrpar{a_0-\dim\ker_{-}\lrpar{\Delta_{\Sigma}}}
	-2C^{-\frac12}a_{-\frac12} \\
	&-C^{-1}a_{-1} + \ta_{-\frac12}C^{-\frac12}(-4-2\log C)
\end{split}\label{gamma.2}\end{equation}
by using the fact that $\frac1{\Gamma\lrpar z}\rest{z=0}=0$, $\pa_z\rest{z=0}\frac1{\Gamma\lrpar z}=1$ and $\pa_z\rest{z=0}\frac1{\Gamma\lrpar{z+1}}=\gamma_e$ is Euler's gamma constant.

More generally, and to connect with \eqref{Sarnak}, we can consider the determinant of 
$\Delta_{\Sigma}$ with its spectrum shifted by a complex number $w$, that is, the determinant
of $\Delta_{\Sigma}+w$. Just as before we have
\begin{equation}\begin{split}
	\zeta_{\Delta_{\Sigma}}\lrpar{z;w} &=  
	\frac1{\Gamma\lrpar z}
	\int_0^\infty t^z \lrpar{{}^R\Tr\lrpar{e^{-t\Delta_{\Sigma}}}-f_0\lrpar t} e^{-tw} \;\frac{dt}t\\
	&+\frac{a_0}{w^z} + \frac{\ta_{-\frac12}}{w^{z-\frac12}}\frac{\left(\Gamma_{\log}\lrpar{z-\frac12} -
	           \log w \Gamma\lrpar{z-\frac12}\right)}{\Gamma(z)} \\
	&+\frac{a_{-\frac12}}{w^{z-\frac12}}\frac{\Gamma\lrpar{z-\frac12}}{\Gamma\lrpar z}
	+\frac{a_{-1}}{w^{z-1}}\lrpar{z-1}^{-1}
\end{split}
\label{gamma.1}\end{equation}
where the function $\Gamma_{\log}(z)$ is defined to be
\begin{equation}
  \Gamma_{\log}(z):= \int_{0}^{\infty} t^{z} e^{-t}\log t \frac{dt}{t}
\label{gamma_log.1}\end{equation}
for $\Re z>0$.  Since it satisfies the recurrence relation
\[
    \Gamma_{\log}(z+1)= z\Gamma_{\log}(z)+ \Gamma(z),
\]
it has a meromorphic continuation to the whole complex plane with poles at 
$-\bbN_{0}=0,-1,-2\ldots$.  In particular, it has no pole at $z=-\frac12$.  Taking the derivative
of $\zeta_{\Delta_{\Sigma}}(z;w)$ with respect to $z$ and setting $z=0$, we get
\begin{equation}\begin{split}
	\zeta_{\Delta_{\Sigma}}'\lrpar{0;w} &=  -\log\det\lrpar{\Delta_{\Sigma}+w} \\
	&=\int_0^\infty  \lrpar{{}^R\Tr\lrpar{e^{-t\Delta_{\Sigma}}}-f_0\lrpar t} e^{-tw} \;\frac{dt}t\\
	&-a_0\log w
	-2\sqrt\pi a_{-\frac12}\sqrt w
	+a_{-1}w\lrpar{-1+\log w} \\
	&+\ta_{-\frac12}\sqrt{w}\left( \Gamma_{\log}(-\frac12)- \log w \Gamma(-\frac12)\right).
\end{split}
\label{der_zeta.1}\end{equation}

\subsection{Relation with the Selberg Zeta function} \label{SZf.0} $ $ \newline
To relate the determinant with the Selberg Zeta function, we will follow \cite{BJP} and use a description of the Selberg Zeta function
in terms of the resolvent of the Laplacian.  Given a hyperbolic surface
$\Sigma$ of genus $g$ with $n$ cusps, we denote by
\begin{equation}
   R_{\Sigma}(s):= \left( \Delta_{\Sigma}+ s(s-1)  \right)^{-1}
\label{SZf.1}\end{equation}
the resolvent of the geometric Laplacian $\Delta_{\Sigma}$  with respect to the hyperbolic metric.  The
Schwartz kernel of $R_{\Sigma}(s)$ is singular along the diagonal.  A natural
way to remove this singular part is to subtract the resolvent of the 
model hyperbolic space
\begin{equation}
  R_{\bbH}(s)= \left( \Delta_{\bbH}+ s(s-1)  \right)^{-1}.
\label{SZf.2}\end{equation}
This resolvent has Schwartz kernel defined on $\bbH\times \bbH$.  Hence, thinking of $\Sigma$ as the quotient  $\Gamma\setminus \bbH$ of the hyperbolic half-plane by some appropriate discrete subgroup $\Gamma\subset \SL(2,\bbR)$, there is
a natural lift of the Schwartz kernel $G_{\Sigma}(s;z,w)$ of $R_{\Sigma}(s)$ to
$\bbH\times\bbH$.  Since locally $R_{\Sigma}(s)$ and $R_{\bbH}(s)$ have the
same full symbol (being a parametrix for $\Delta_{\bbH}+s(s-1)$), 
they will have the same singularities along the diagonal.  This means the 
function
\begin{equation}
   \varphi_{\Sigma}(s;z):= (2s-1)\left[  
              G_{\Sigma}(s;z,w)-G_{\bbH}(s;z,w)\right]_{w=z}
\label{SZf.3}\end{equation}
will be smooth in $z\in \bbH$ and meromorphic in $s$.  Because of the
$\SL(2,\bbR)$ invariance of $\Delta_{\bbH}$, the Schwartz kernel 
$G_{\bbH}$ will also be $\SL(2,\bbR)$ invariant when restricted to the 
diagonal in $\bbH\times \bbH$.  This means in particular that the function
$\varphi_{\Sigma}(s;z)$ will be $\Gamma$-invariant and so will descend to give
a function on $\Sigma= \Gamma\setminus \bbH$.  Thus, we can define the function
\begin{equation}
 \phi_{\Sigma}(s):= \sideset{^R}{}\int_{\Sigma} \varphi_{\Sigma}(s;z) dg_{\Sigma}(z)
\label{SZf.4}\end{equation}
where $dg_{\Sigma}$ is the volume form associated to the hyperbolic metric, and the integral is renormalized using $\rho_\Sigma$, the boundary defining function of \eqref{bc.10}.  This function has a meromorphic continuation to the complex plane with possible poles at $\frac{1}{2}- \bbN_{0}/2$.
As a particular example, we can consider the Horn 
$H:=\Gamma_{\infty}\setminus \bbH$ of \eqref{bc.4}.  The end obtained
as $y\to +\infty$ is a cusp end, so we can pick a boundary defining function
as usual there, but the other end when $y\to 0^{+}$ is not a cusp end,
but a funnel, and for that end, one should take a boundary defining function
which is given by $y$ near $y=0$.  With this choice, we can make sense of 
the function $\phi_{H}(s)$. The following proposition is due to Borthwick, Judge, and Perry \cite{BJP}.

\begin{proposition}[\cite{BJP}, Proposition 4.3]
Let $\Sigma$ be a Riemann surface of genus $g$ and with $n$ cusps.  Then 
\[
     \frac{Z_{\Sigma}'(s)}{Z_{\Sigma}(s)}= \phi_{\Sigma}(s)-n \phi_{H}(s).
\]
\label{SZf.5}\end{proposition}
The function $\phi_H(s)$ can be computed explicitly \cite[Proposition 2.4]{BJP},
\begin{equation}
  \phi_{H}(s)= -\log 2 - \Psi \lrpar{s+\frac{1}{2}}+ \frac{1}{2s-1}
\label{SZf.6}\end{equation}
where $\Psi(z)$ is the digamma function $\frac{\Gamma'(z)}{\Gamma(z)}$.
Thus, to relate the logarithmic derivative of $Z_{\Sigma}(s)$ with the 
determinant we need to understand the function $\phi_{\Sigma}(s)$ in terms
of the heat kernel instead of the resolvent.  
\begin{lemma}
In the sense of distributions, we have that
\[
    (\Delta_{\Sigma}+ s)^{-1}= \int_{0}^{\infty}K_{\Sigma}(t)e^{-st}dt,
\quad \mbox{for}\; \Re s>0.
\]
\label{SZf.7}\end{lemma}  
\begin{proof}
Let $f\in \CI_{c}(\Sigma)$ be a test function.  Then by definition of the
heat kernel, we have that
\[
     \pa_{t}K_{\Sigma}(t)f + \Delta_{\Sigma}K_{\Sigma}(t)f=0, \quad
   K_{\Sigma}(0)f=f.
\]
Using integration by part, this implies that 
\begin{equation}
\begin{aligned}
\Delta_{\Sigma} \left( \int_{0}^{\infty}K_{\Sigma}(t)f e^{-st}dt \right)&=
   \int_{0}^{-\infty} -\pa_{t}K_{\Sigma}(t)f e^{-st}dt \\
&= \left. -K_{\Sigma}(t)f e^{-st}\right|_{0}^{\infty}- s \int_{0}^{\infty} 
K_{\Sigma}(t)f e^{-st}dt \\
&= f- s \int_{0}^{\infty} K_{\Sigma}(t)f e^{-st}dt,
\end{aligned}
\label{SZf.8}\end{equation}
which shows that 
\[
          (\Delta_{\Sigma}+s) \left( \int_{0}^{\infty}K_{\Sigma}(t)f e^{-st}dt \right)= f.
\]
This means that
\[
          f\mapsto \int_{0}^{\infty} K_{\Sigma}(t)f e^{-st}dt
\]
is a right inverse for $(\Delta_{\Sigma}+s)$.  The same computation shows that
it is a left inverse since
\[
             K_{\Sigma}(t)\Delta_{\Sigma}f= \Delta_{\Sigma}K_{\Sigma}(t)f
\]
by uniqueness of the solution for the heat equation.
\end{proof}
From this lemma we conclude that 
\begin{equation}
 \frac{\varphi_{\Sigma}(s;z)}{2s-1}= \int_{0}^{\infty} \left(K_{\Sigma}(t,z,z)- K_{\bbH}(t;z,z)\right) e^{-s(s-1)t} dt.
\label{SZf.9}\end{equation}

Since the left-hand side is smooth, the right-hand side is smooth as well, which means that $K_{\Sigma}(t;z,z)$ and $K_{\bbH}(t;z,z)$ have the same term of order $t^{-1}$ in their asymptotic expansions as $t\searrow 0$.
Integrating \eqref{SZf.9} in $z$ and taking the finite part, we get 
\begin{equation}
  \frac{ \phi_{\Sigma}(s)}{2s-1}= \sideset{^R}{}\int_{\Sigma} \int_{0}^{\infty} 
            \left( K_{\Sigma}(t;z,z)- K_{\bbH}(t;z,z)  \right) e^{-ts(s-1)}dt dg_{\Sigma}(z).
\label{SZf.10}\end{equation}
The order of integration can be interchanged, since for $\Re(w)\gg 1$,
\begin{multline}
\int_{\Sigma} \int_{0}^{\infty} 
            x^{w}\left( K_{\Sigma}(t;z,z)- K_{\bbH}(t;z,z)  \right) e^{-ts(s-1)}dt dg_{\Sigma}(z)= \\
            \int_{0}^{\infty} \int_{\Sigma} 
            x^{w}\left( K_{\Sigma}(t;z,z)- K_{\bbH}(t;z,z)  \right) e^{-ts(s-1)}dt dg_{\Sigma}(z).
\label{SZf.11}\end{multline}
Hence, we get that
\begin{equation}
  \frac{\phi_{\Sigma}(s)}{2s-1}= \int_{0}^{\infty} \left( {}^{R}\Tr(K_{\Sigma}(t)) - 
         {}^{R}\Tr_{\Sigma}(K_{\bbH}(t))\right) e^{-ts(s-1)}dt
\label{SZf.12}\end{equation}
where 
\[
    {}^{R}\Tr_{\Sigma}(K_{\bbH}(t)):= \FP\int_{\Sigma} K_{\bbH}(t;z,z)dg_{\Sigma}(z).
\]
We recall that $K_{\bbH}(t)$ itself does not descend to $\Sigma \times \Sigma$, but its restriction to the diagonal in $\bbH \times \bbH$ does descend to the diagonal in $\Sigma \times \Sigma$. Indeed, it is well-known that because the hyperbolic metric on $\bbH$ is $\SL(2,\bbR)$-invariant, $K_{\bbH}(t;z,z)$ is
constant in $z$, say equal to $k_{\bbH}(t)dg_{\bbH}(z)$ for some function of $t$.
Since the surfaces we are studying have finite area, we have
\[
      {}^{R}\Tr( K_{\bbH}(t))= k_{\bbH}(t) \area(\Sigma).
\]
\begin{corollary}
The function $k_{\bbH}(t)$ 
has an asymptotic expansion given by
\[
     k_{\bbH}(t) \sim k_{-1}t^{-1}+ o(t^{-1})\quad \mbox{as}\;\, t\searrow 0.
\]
Therefore, if $\Sigma$ is a Riemann surface of genus $g$ with $n$ cusps, then the regularized trace
of its heat kernel has the asymptotic expansion
\[
   {}^{R}\Tr(K_{\Sigma}(t)) \sim k_{-1}\area(\Sigma) t^{-1}+\mathcal{O}(
t^{-\frac{1}{2}}\log t) \quad\mbox{as}\;\, t\searrow 0.
\]
\label{SZf.14}\end{corollary}
\begin{proof}
Consider the case where $\Sigma$ has no cusp.  Then it is well-known that 
${}^{R}\Tr(K_{\Sigma}(t))= \Tr(K_{\Sigma}(t))$ has asymptotic expansion of the form
\[
    \Tr(K_{\Sigma}(t)) \sim \alpha t^{-1} + \beta + \mathcal{O}(t) \quad \mbox{as}\;\, t\searrow 0,
\]
where $\alpha$ and $\beta$ are some constants.  By formula \eqref{SZf.12}, ${}^{R}\Tr(K_{\bbH}(t))$
has the same term of order $t^{-1}$ as $t\searrow 0$, hence we get that
\[
    k_{-1}:= \frac{\alpha}{\area(\Sigma)}= \frac{1}{4\pi}
\]
is such that 
\[
     k_{\bbH}(t) \sim k_{-1}t^{-1}+o(t^{-1})\quad \mbox{as}\;\, t\searrow 0.
\]
When we consider a surface with cusps, ${}^{R}\Tr(K_{\Sigma})$ will have the same term of order $t^{-1}$ as
${}^{R}\Tr_{\Sigma}(K_{\bbH}(t))$ as $t\searrow 0$, hence
\[
   {}^{R}\Tr(K_{\Sigma}(t))\sim k_{-1}\area(\Sigma) t^{-1} + 
          \mathcal{O}(t^{-\frac{1}{2}}\log t), \quad
     \mbox{as}\;\, t\searrow 0.
\]
\end{proof}

Consider the functional 
\begin{equation}
D_{\Sigma}\lrpar s := \det\lrpar{\Delta_{\Sigma}+s\lrpar{s-1}}= \exp\left(-\zeta_{\Delta_{\Sigma}}'(0;s(s-1)) \right).
\label{D.1}\end{equation}
If we differentiate with respect $s$ and use formula \eqref{der_zeta.1}, we  find that 
\begin{equation}\begin{split}\label{D'D}
	\frac{1}{2s-1}\frac{D_{\Sigma}'\lrpar s}{D_{\Sigma}\lrpar s}
	&=\int_0^\infty  \lrpar{{}^R\Tr\lrpar{e^{-t\Delta_{\Sigma}}}-a_{-1}t^{-1}} e^{-ts\lrpar{s-1}} \;dt\\
	&-a_{-1}\log \lrpar{ s(s-1)} .
\end{split}\end{equation}

Combining formula \eqref{D'D} and \eqref{SZf.12} and using corollary~\ref{SZf.14}, we get that
\begin{multline}\label{SZf.17}
\frac{1}{2s-1}\left( \frac{D_{\Sigma}'(s)}{D_{\Sigma}(s)}- \frac{Z'_{\Sigma}(s)}{Z_{\Sigma}(s)}\right) =\\
  \int_{0}^{\infty} \left( {}^{R}\Tr(K_{\Sigma}(t))- a_{-1}t^{-1} -{}^{R}\Tr(K_{\Sigma}(t)) +
          {}^{R}\Tr_{\Sigma}(K_{\bbH}(t)) \right)e^{-ts(s-1)}dt \\
          -k_{-1}\area(\Sigma)\log \left(s(s-1)\right) + \frac{n}{2s-1} \phi_{H}(s)  \\
       = \int_{0}^{\infty} \left( {}^{R}\Tr(K_{\bbH}(t)) -k_{-1}\area(\Sigma)t^{-1}  \right)e^{-ts(s-1)}dt  \\
        \hspace{1cm} - k_{-1}\area(\Sigma)\log \left(s(s-1)\right) +\frac{n}{2s-1} \phi_{H}(s),
\end{multline}
so that
\begin{multline}
\frac{1}{2s-1}\left( \frac{D_{\Sigma}'(s)}{D_{\Sigma}(s)}- \frac{Z'_{\Sigma}(s)}{Z_{\Sigma}(s)}\right)= 
\frac{n}{2s-1} \phi_{H}(s) \\
       \hspace{1cm} + \area(\Sigma) \left[
       \int_{0}^{\infty} \left( k_{\bbH}(t) -k_{-1}t^{-1}  \right)e^{-ts(s-1)}dt
       - k_{-1}\log\left( s(s-1)\right) \right].   
\label{SZf.18}\end{multline}
In particular, for fixed $g$ and $n$, we see that the right hand side does not depend on the
$\Sigma$ since the area is given by $- 2\pi \chi(\Sigma)$ by the Gauss-Bonnet theorem, a quantity that
only depends on $g$ and $n$.  

From \eqref{SZf.6}, we see that 
\begin{equation}
   \phi_{H}(s)= \frac{d}{ds} \left( -s\log 2 - \log\Gamma(s+\frac12) + \frac12\log(2s-1)  \right).
\label{as.1}\end{equation}
Then, according to \eqref{SZf.18} and \eqref{Sarnak}, there exists a constant $C$ such that
\begin{equation}
   \frac{D_{\Sigma}(s)}{Z_{\Sigma}(s)}= C \left(e^{E}-s(s-1) \frac{\Gamma_{2}(s)^{2}}{\Gamma(s)} (2\pi)^{s}
                                          \right)^{-\chi(\Sigma)} 
                                          \left( \frac{\sqrt{2s-1}}{2^{s}\Gamma(s+\frac12)}\right)^{n}
\label{as.2}\end{equation}
where $E:= -\frac{1}{4}-\frac12\log 2\pi + 2 \zeta'(-1)$.  As in \cite{Sarnak}, the constant $C$ can be determined by the 
asymptotic expansion of the logarithm of both sides of \eqref{as.2} as $s$ approaches infinity. 
For the left side, it is clear from \eqref{Selberg.1} that $\log Z_{\Sigma}(s)$ has a trivial asymptotic
expansion as $s\to +\infty$.  Thus, from \eqref{der_zeta.1}, we conclude that the asymptotic behavior of the
logarithm of the left side is given by
\begin{equation}
\begin{split}
\log\lrpar{ \frac{D_{\Sigma}(s)}{Z_{\Sigma}(s)}}&= a_{0}\log s(s-1) + 2\sqrt{\pi}a_{-\frac12}\sqrt{s(s-1)}\\
  &- \ta_{-\frac12}\sqrt{s(s-1)}\lrpar{ \Gamma_{\log}(-\frac12)+ 2\sqrt{\pi}\log(s(s-1))}\\
  &- a_{-1}s(s-1)\lrpar{-1+\log(s(s-1))} + o(1)     
\end{split}
\label{as.3}\end{equation} 
as $s\to+\infty$.  On the other hand, if we set 
\begin{equation}
Z_{\cusp}(s)= \frac{\sqrt{2s-1}}{2^{s}\Gamma(s+\frac12)}= \frac{\sqrt{2}}{2^{s}\sqrt{s-\frac12}
          \Gamma(s-\frac12)},
\label{as.4}\end{equation}
we see using Stirling's formula that its logarithm has the following asymptotic behavior,
\begin{equation}
\log \lrpar{Z_{\cusp}(s)}= -\frac12\log(2\pi) + \lrpar{1-\log 2}\lrpar{s-\frac12}
          -\lrpar{s-\frac12}\log\lrpar{s-\frac12}+ o(1)
\label{as.5}\end{equation}
as $s\to +\infty$.  Since
\begin{gather}
   \sqrt{s(s-1)}= s-\frac12 + o(1),  \\
   \sqrt{s(s-1)} \log \lrpar{s(s-1)}= 2\lrpar{s-\frac12} \log\lrpar{s-\frac12}
   + o(1),
\label{as.6}\end{gather}
as $s\to +\infty$, we can rewrite \eqref{as.3} as
\begin{equation}
\begin{split}
\log\lrpar{ \frac{D_{\Sigma}(s)}{Z_{\Sigma}(s)}}&= a_{0}\log s(s-1) + 
    \lrpar{ 2\sqrt{\pi} a_{-\frac12}- \ta_{-\frac12}\Gamma_{\log}(-\frac12)}\lrpar{s-\frac12}\\
  &- 4\sqrt{\pi}\ta_{-\frac12}\lrpar{s-\frac12}\log\lrpar{s-\frac12} \\
  &- a_{-1}s(s-1)\lrpar{-1+\log(s(s-1))} + o(1)     
\end{split}
\label{as.7}\end{equation} 
as $s\to +\infty$.  Now, from \cite{Sarnak}\footnote{In $(2.19)$ of \cite{Sarnak}, the coefficient of 
$\log s(s-1)$ is $\frac{1}{24}$, but it is supposed to be $\frac{1}{6}$}, we have also that 
\begin{multline}
\log\lrpar{e^{E-s(s-1)}\frac{\Gamma_{2}(s)^{2}}{\Gamma(s)}(2\pi)^{s}}= \\
 -\frac{1}{6}\log s(s-1) 
   +\frac12 s(s-1) - \frac{s(s-1)}{2}\log s(s-1) + o(1)
\label{as.22}\end{multline}
as $s\to +\infty$.  This asymptotic behavior only involves
terms of the form $\log\lrpar{s(s-1)}$ and $s(s-1)\log\lrpar{s(s-1)}$.  Thus, in \eqref{as.7}, the terms
involving $a_{-\frac12}$ and $\ta_{-\frac12}$ counterbalance the asymptotic behavior of
\eqref{as.5} while the terms involving $a_{-1}$ and $a_{0}$ counterbalance the asymptotic
behavior of \eqref{as.22}.
Comparing \eqref{as.7} with \eqref{as.22}, we find 
\begin{equation}
   a_{-1}= g-1= -\frac{\chi(\Sigma)}{2}, \quad a_{0}= \frac{\chi(\Sigma)}{6}.
\label{as.8}\end{equation}
Comparing \eqref{as.7} with \eqref{as.5}, we also get
\begin{equation}
   \ta_{-\frac12}= \frac{n}{4\sqrt{\pi}}, \quad a_{-\frac12}= \frac{n}{2\sqrt{\pi}} 
                                                 \lrpar{ 1-\log 2 + \frac{\Gamma_{\log}(-\frac12)}{4\sqrt{\pi}}}.
\label{as.9}\end{equation}
Now, recall that in \cite{Sarnak}, the constant $E$ is chosen so that 
\begin{multline}
  \log \lrpar{e^{E-s(s-1)}\frac{\Gamma_{2}(s)^{2}}{\Gamma(s)}(2\pi)^{s}}= \\
     -a_{-1}s(s-1)\lrpar{-1 + \log \lrpar{s(s-1)}} +a_{0}\log\lrpar{s(s-1)}.
\end{multline}
This means the constant $C$ has to be chosen to compensate the constant term of \eqref{as.5}, that is,
\begin{equation}
\log C= -\frac{n}{2} \log 2\pi \quad \Longrightarrow \quad C= (2\pi)^{-\frac{n}{2}}.
\label{as.10}\end{equation}
This gives the following result.
\begin{theorem}
For a Riemann surface of genus $g$ with $n$ cusps satisfying $2g-2+n>0$ and
equipped with the hyperbolic metric, we have
\[
  \det\lrpar{\Delta_{\Sigma}+s(s-1)}= Z_{\Sigma}(s)
    \frac{\lrpar{e^{E-s(s-1)}\frac{\Gamma_{2}(s)^{2}}{\Gamma(s)}(2\pi)^{s}}^{-\chi(\Sigma)}}
         {\lrpar{2^{s}\sqrt{\pi(s-\frac12)}\Gamma(s-\frac12)}^{n}}.
\]
\label{as.11}\end{theorem}
As a consequence, we see that the ratio 
\[
                  \frac{\det\lrpar{\Delta_{\Sigma}+s(s-1)}}{Z_{\Sigma}(s)}
\]
is a meromorphic function in $s$ which only depends on the genus $g$ and the number of cusps $n$.
This means that, up to a multiplicative constant depending only on $g$ and $n$, the determinant
of $\Delta_{\Sigma}$ is given by $Z_{\Sigma}'(1)$.  The formula of theorem~\ref{as.11} can also be
expressed in terms of the $\overline{\pa}$-Laplacian
\[
            \Delta_{\bpa}= \frac{1}{2}\Delta_{\Sigma}.
\]
For the heat kernel of the $\bpa$-Lapacian, we have the following short time asymptotic expansion,
\begin{equation}
\begin{split}
{}^{R}\Tr\lrpar{e^{-t\Delta_{\bpa}}} &= {}^{R}\Tr\lrpar{ e^{-\frac{t}{2}\Delta_{\Sigma}}}  \\
                   &= \frac{2a_{-1}}{t}+ \frac{\sqrt{2}\ta_{-\frac12}}{\sqrt{t}}\log t  
                   + \frac{\sqrt{2}a_{-\frac12}-\sqrt{2}\ta_{-\frac12}\log 2}{\sqrt{t}} + a_{0} + \mathcal{O}(\sqrt{t})
\end{split}
\label{bpa.1}\end{equation}
as $t\to 0^{+}$, where $a_{-1}, \ta_{-\frac12},a_{-\frac12},a_{0}$ are the coefficients in
\eqref{expansion.1}.  From formula \eqref{gamma.1}, we conclude that
\[
    \zeta_{\Delta_{\bpa}}\lrpar{z;\frac{s(s-1)}{2}}= 2^{z} \zeta_{\Delta_{\Sigma}}\lrpar{ z;s(s-1)}.
\] 
Hence,
\[
   \zeta'_{\Delta_{\bpa}}\lrpar{0;\frac{s(s-1)}{2}}= \lrpar{\log 2} \zeta_{\Delta_{\Sigma}}\lrpar{0;s(s-1)}
      + \zeta_{\Sigma}'\lrpar{0;s(s-1)},
\]
which means that 
\[
   \det\lrpar{ \Delta_{\bpa}+ \frac{s(s-1)}{2}}= 2^{-\zeta_{\Delta_{\Sigma}}(0;s(s-1))}
                                                  \det\lrpar{ \Delta_{\Sigma}+s(s-1)}.
\]
Now, $\zeta_{\Delta_{\Sigma}}\lrpar{0;s(s-1)}$ can be computed explicitly from \eqref{gamma.1}
and \eqref{as.8},
\begin{equation}
\begin{split}
\zeta_{\Delta_{\Sigma}}\lrpar{0;s(s-1)} &= a_{0} -a_{-1} s(s-1)  \\
                                        &= \chi(\Sigma) \lrpar{ \frac{1}{6}+ \frac{s(s-1)}{2}}.
\end{split}
\label{bpa.2}\end{equation}
From theorem~\ref{as.11}, we get the following formula.
\begin{corollary}
For a Riemann surface $\Sigma$ of genus $g$ with $n$ cusps satisfying $2g-2+n>0$ and equipped with the hyperbolic metric, we have
\[
\det\lrpar{\Delta_{\bpa}+\frac{s(s-1)}{2}}= Z_{\Sigma}(s)
\frac{\lrpar{ 2^{\frac{1}{6}+\frac{s(s-1)}{2}} e^{E-s(s-1)}\frac{\Gamma_{2}(s)^{2}}{\Gamma(s)}(2\pi)^{s}}^{-\chi(\Sigma)}}
         {\lrpar{2^{s}\sqrt{\pi(s-\frac12)}\Gamma(s-\frac12)}^{n}}.
\]  
\label{bpa.3}\end{corollary}

\subsection{The determinant of $\Delta_{\ell}$ for $\ell\ge 1$}\label{lge.0} $ $\newline

As indicated in \cite{Sarnak}, it is possible to express the determinant of $\Delta_{\ell}$ in terms of 
Selberg Zeta function by using corollary~\ref{bpa.3}.  This is because the spectrum of $\Delta_{\ell}$ is 
essentially given by a shifted version of the spectrum of $\Delta_{0}=\bpa^{*}\bpa$.

Recall first that these various Laplacians are related by the recurrence relation (see\footnote{We have $\ell-1$ instead of $\frac{\ell-1}{2}$ in \cite{TZ} since we use the convention
$|dz|^{2}=2$.} for instance
(1.3) in \cite{TZ})
\begin{equation}
   \Delta_{\ell}\bpa^{*}_{\ell}u = \bpa_{\ell}^{*}u \lrpar{ \Delta_{\ell-1}+\ell-1}
\label{lge.1}\end{equation} 
where $u:=\frac{1}{y^{2}}$ is seen as a section of $\Lambda^{1,0}_{\Sigma}
\otimes \Lambda^{0,1}_{\Sigma}$ on
$\Sigma$.  Taking the formal adjoint
of \eqref{lge.1}, we get
\begin{equation}
   u^{*} \bpa_{\ell}\Delta_{\ell}= \lrpar{ \Delta_{\ell-1}+\ell-1} u^{*} \bpa_{\ell}
\label{lge.2}\end{equation} 
where $u^{*}$ is the conjugate $\overline{u}$ of $u$ seen as a section of
$(\Lambda_{\Sigma}^{1,0})^{-1}\otimes (\Lambda_{\Sigma}^{0,1})^{-1}$.

Since the operator $\db_{\ell}$ is Fredholm, it has a well-defined parametrix
$\db_{\ell}^{-1}: (\ker \db^{*}_{\ell})^{\perp}\to (\ker \db_{\ell})^{\perp}$.  Applying this
parametrix to both sides of \eqref{lge.2}, we get
\begin{equation}
   \Delta_{\ell}=\lrpar{u^{*} \bpa_{\ell}}^{-1} \lrpar{\Delta_{\ell-1}+\ell-1}\lrpar{u^{*} \bpa_{\ell}}.
\label{lge.3}\end{equation}
In the compact case, this directly implies that
\begin{equation}
  \det'(\Delta_{\ell})= \det( \Delta_{\ell-1} +\ell-1)
\label{lge.4}\end{equation}
for $\ell\ge 2$ since $\db_{\ell}$ is surjective in that case.  When $\ell=1$, the operator
$\db_{1}$ is not surjective, but the cokernel of $u^{*}\db_{\ell}$,
\[
       \ker \db^{*}_{1}u= \ker \db_{0}= \ker \Delta_{0}
\]
is precisely the kernel of $\Delta_{0}$, so that we have in that case
\[
             \det'(\Delta_{1})= \det'(\Delta_{0}).
\]
Using \eqref{lge.3} once more and \eqref{rr.24}, we have on the other hand that
for $k>0$, 
\begin{equation}
   \det\lrpar{\Delta_{\ell-1}+ k} =
   \left\{ \begin{array}{ll}  \lrpar{k}^{g-1} \det\lrpar{ \Delta_{0}+ k}, & \ell=2; \\
                   \lrpar{k}^{(2\ell-1)(g-1)} \det\lrpar{\Delta_{\ell-2}+ k
                              +\ell-2 }, & \ell\ge 3.   
           \end{array} \right.
\label{lge.5}\end{equation}
Applying this recursively, we get
\begin{equation}
  \det'(\Delta_{\ell})= \left\{  \begin{array}{ll}
        \det'(\Delta_{0}), & \ell=1;  \\
         \det\lrpar{\Delta_{0}+1}, & \ell=2; \\
        \delta_{\ell,g}\det\lrpar{ \Delta_{0}+ \ell(\ell-1)},&
                               \ell\ge 3.
      \end{array}\right.
\label{lge.6}\end{equation}
where $\delta_{\ell,g}$ is a number depending only on $\ell$ and $g$.  In the non-compact case, one has to be more careful since the regularized trace does not necessarily
vanish on a commutator.  
Taking this into account, the analog of \eqref{lge.3} in the non-compact case is
\begin{equation}
  \det'(\Delta_{\ell})= D_{\ell,n} \det\lrpar{ \Delta_{\ell-1} +\ell-1}
\label{lge.7}\end{equation}
with
\begin{equation}
   -\log \lrpar{D_{\ell,n}}= \lrpar{ \frac{d}{dz} \frac{1}{\Gamma(z)}\int_{0}^{\infty}
       t^{z} {}^{R}\Tr\lrpar{ [(u^{*}\bpa_{\ell})^{-1}e^{-t\Delta_{\ell-1}}, u^{*} \bpa_{\ell}] }e^{-\frac{t(\ell-1)}{2}} \frac{dt}{t}}_{z=0}
\label{lge.8}\end{equation} 
regularizing as in \eqref{gamma.2}.
Although the term $D_{\ell,n}$ might be hard to compute, what is clear is that it only depends on 
$\ell$ and the number $n$ of cusps.  
This is because the regularized trace of a commutator $[A,B]$ `localizes' near the boundary in the sense that it only depends on the Taylor expansion of the integral kernels at the boundary of the diagonal. 
Recall that to construct the heat kernel (see \cite{Vaillant, Albin-Rochon}) we start with a `parametrix' for the heat equation which solves a model equation at the cusp. The solution of this model equation is then used iteratively to construct the Taylor expansion of the heat kernel as we approach the cusp, before finally solving away the remaining error in the interior. The upshot is that, since all cusps have isometric neighborhoods, the term $D_{\ell,n}$ only depends on $\ell$ and $n$ as required. 

Thus using recursively \eqref{lge.7} and applying corollary~\ref{bpa.3}, we get the following.
\begin{corollary}
For a Riemann surface $\Sigma$ of genus $g\ge 2$ with $n$ cusp, we have
\[
   \det'(\Delta_{\ell}) = \left\{ \begin{array}{ll}
         \alpha_{\ell,g,n} Z_{\Sigma}(\ell), &   \ell\ge 2;\\
         \alpha_{\ell,g,n}Z_{\Sigma}'(1), & \ell=0,1;
   \end{array}\right. 
\]
where each constant $\alpha_{\ell,g,n}>0$ only depends on $\ell$, $g$ and $n$.  A similar 
statement holds of the determinant of $D^{-}_{\ell}D^{+}_{\ell}= 2 \Delta_{\ell}$.
\label{lge.9}\end{corollary}

\section{The curvature of the Quillen connection}\label{qc.0} $ $\newline
Recall that the determinant bundle of the family of operator $\bpa_{\ell}$ is by definition
\begin{equation}
 \lambda_{\ell}:= \det \ind \bpa_{\ell}= \Lambda^{\max}\ker \bpa_{\ell}\otimes 
            \lrpar{\Lambda^{\max}\coker \bpa_{\ell}}^{-1}
\label{qc.1}\end{equation}
where $\ell\in\bbZ$ and $\Lambda^{\max}$ denotes the maximal exterior power of a vector space.  The 
definition is particularly simple in this case because $\ker \bpa_{l}$ is a vector bundle over
$T_{g,n}$.  The $L^{2}$-norm on $\ker\bpa_{\ell}$ defines a canonical metric on $\lambda_{\ell}$,
the $L^{2}$-metric, denoted $\|\cdot\|$.  An alternative metric which is more interesting geometrically is 
the \textbf{Quillen metric},
\begin{equation}
   \|\cdot\|_{Q} := \lrpar{\det D^{-}_{\ell}D^{+}_{\ell}}^{-\frac12}\|\cdot\|.
\label{qc.2}\end{equation}
Following the discussion of \S~9.7 in \cite{BGV}, we will associate to $\|\cdot\|_{Q}$ a compatible connection called the \textbf{Quillen connection}.  In order to do that, consider over
$\cT_{g,n}$ the $\bbZ_{2}$-graded bundle
\begin{equation}
  \cE_{\ell}= \cE^{+}_{\ell}\oplus \cE^{-}_{\ell}, \quad
                      \cE^{+}_{\ell}:= \Lambda^{\ell,0}\lrpar{\cT_{g,n}/T_{g,n}}, \quad
                      \cE^{-}_{\ell}:= \Lambda^{\ell,1}\lrpar{\cT_{g,n}/T_{g,n}}. 
\label{qc.3}\end{equation}  
Let also $\pi_{*}\cE_{\ell}\to T_{g,n}$ be the Fr\'echet bundle whose fiber at $[\Sigma]\in T_{g,n}$ is
\begin{equation}
    \pi_{*}\cE_{\ell,[\Sigma]}:= \dot{\cC}^{\infty}\lrpar{\Sigma; \left.\cE_{l}\right|_{\Sigma}\otimes 
     |\Lambda_{\Sigma}|^{\frac{1}{2}}},
\label{qc.4}\end{equation}
where $|\Lambda_{\Sigma}|$ is the density bundle on $\Sigma$ and
$\dot{\cC}^{\infty}\lrpar{\Sigma; \left.\cE_{l}\right|_{\Sigma}\otimes 
     |\Lambda_{\Sigma}|^{\frac{1}{2}}}$ is the space of smooth
sections of $\left. \cE_{l}\right|_{\Sigma}\otimes 
     |\Lambda_{\Sigma}|^{\frac{1}{2}}$
with rapid decay at infinity.
The family of Dirac type operators 
\begin{equation}
   D_{\ell}:= \sqrt{2}\lrpar{ \bpa_{\ell}+ \bpa_{\ell}^{*}}, \quad D^{+}_{\ell}=\sqrt{2}\bpa_{\ell},
   \quad D^{-}_{\ell}= \sqrt{2}\bpa_{\ell}^{*},
\label{qc.5}\end{equation}
acts from $\pi_{*}\cE_{\ell}$ to $\pi_{*}\cE_{\ell}$.  One of the reasons that motivates
the introduction of the fibre density bundle in the definition of $\pi_{*}\cE$ is that in this
way the canonical connection on
$\pi: \cT_{g,n}\to T_{g,n}$ induces a connection on $\pi_{*}\cE_{\ell}$, denoted
$\nabla^{\pi_{*}\cE_{\ell}}$, which is automatically compatible with the metric
of $\pi_{*}\cE_{\ell}$ (\cf proposition 9.13 in \cite{BGV}).  Notice also that the
density bundle $|\Lambda_{\Sigma}|$ is canonically trivialized by the section
$|dg_{\Sigma}|$ so that $D_{\ell}$ acts on $\pi_{*}\cE_{\ell}$ in a natural way. 
To the family of Dirac type operators $D_{\ell}$, we can associate
a superconnection
\begin{equation}
   \bbA_{\ell}:= D_{\ell} + \nabla^{\pi_{*}\cE_{\ell}}
\label{qc.6}\end{equation}
and its rescaled version
\begin{equation}
   \bbA_{\ell}^{s}:= s^{\frac12}D_{\ell} + \nabla^{\pi_{*}\cE_{\ell}}.
\label{qc.7}\end{equation}
For $s\in \bbR^{+}$, we can define two differential forms $\alpha^{\pm}_{\ell}\in \cA(T_{g,n})(s)$,
\begin{equation}
\begin{split}
  \alpha^{\pm}_{\ell}(s) &:= {}^{R}\Tr_{\pi_{*}\cE^{\pm}_{\ell}}\lrpar{
                   \frac{\pa \bbA_{\ell}^{s}}{\pa s} e^{-(\bbA^{s}_{\ell})^{2}}} \\
                   &= \frac{1}{2s^{\frac12}} {}^{R}\Tr_{\pi_{*}\cE^{\pm}_{\ell}}
                   \lrpar{D_{\ell} e^{-(\bbA^{s}_{\ell})^{2}}},
\end{split}
\label{qc.8}\end{equation}
by taking the trace with respect to $\cE_{\ell}^{+}$ and $\cE_{\ell}^{-}$ respectively.
The $0$-form component of $(\bbA^{2}_{\ell})^{s}$ is $s D_{\ell}^{2}$, while its 
$1$-form component is $s^{\frac12}[\nabla^{\pi_{*}\cE_{l}},D_{\ell}]$.  On the other hand,
the $1$-form component of $e^{-(\bbA^{s}_{\ell})^{2}}$ is given by
\begin{equation}
\begin{split}
\lrpar{ e^{-(\bbA^{s}_{\ell})^{2}}}_{[1]} &= (-s) \int_{0}^{1} e^{-(1-\sigma)sD^{2}_{\ell}} s^{-\frac12}
                         [\nabla^{\pi_{*}\cE_{\ell}},D_{\ell}] e^{-\sigma s D^{2}_{\ell}} d\sigma \\
         &= -s^{\frac12} \int_{0}^{1} e^{-(1-\sigma)sD^{2}_{\ell}}
                         [\nabla^{\pi_{*}\cE_{\ell}},D_{\ell}] e^{-\sigma sD^{2}_{\ell}} d\sigma.
\end{split}
\label{qc.9}\end{equation}
The following observation will turn out to be very useful.
\begin{lemma}
The Schwartz kernel of $[\nabla^{\pi_{*}\cE_{\ell}},D_{\ell}^{\pm}]$ vanishes to all order at the front face.
In particular, for $P\in \Psi^{-\infty}(\cT_{g,n}/T_{g,n};\cE_{\ell})$,
\[
    {}^{R}\STr\lrpar{ [ [\nabla^{\pi_{*}\cE_{\ell}},D_{l}^{\pm}], P]}=0
\]
\label{qc.10}\end{lemma}
\begin{proof}
Let $[\Sigma]\in T_{g,n}$ be given.
If $\mu\in \Omega^{-1,1}(\Sigma)$ is a harmonic Beltrami differential, let 
$f^{\mu}:\bbH\to \bbH$ be the unique diffeomorphism satisfying the Beltrami equation
\[
             \frac{\pa f^{\mu}}{\pa \overline{z}}= \mu \frac{\pa f^{\mu}}{\pa z}
\] 
and fixing the points $0,1,\infty$.  In particular, since $\mu$ is a cusp form, it decreases rapidly
as $z\to \infty$.  This means that $f^{\mu}$ is asymptotically holomorphic as $z\to \infty$.  From the 
definition of the canonical connection on $\pi:\cT_{g,n}\to T_{g,n}$, this means that 
 the Schwartz kernel $[\nabla^{\pi_{*}\cE_{\ell}},D^{+}_{\ell}](z,z')$  decreases quickly as
$z$ and $z'$ approaches a cusp in $\Sigma$.  For $D^{-}_{\ell}=\sqrt{2}\bpa_{\ell}^{*}$, the same is
true, but since $\bpa_{\ell}^{*}= -u^{\ell-1}\pa u^{-\ell}$, we also need to use the fact that
$u= \frac{1}{y^{2}}$ is parallel with respect to the canonical connection on $\pi:\cT_{g,n}\to T_{g,n}$.  
Now, we know that ${}^{R}\STr\lrpar{ [ [\nabla^{\pi_{*}\cE_{\ell}},D_{\ell}^{\pm}], P]}$  
depends linearly on the asymptotic expansions of $[\nabla^{\pi_{*}\cE_{\ell}},D_{\ell}^{\pm}]$
and $P$ at the corner of $\Sigma\times \Sigma$.  The asymptoptic expansion of 
$[\nabla^{\pi_{*}\cE_{\ell}},D_{\ell}^{\pm}]$ being trivial, the result follows.

Alternatively, the result also follows directly from the explicit formulas
\eqref{con.18}.
\end{proof}
With this lemma, the discussion of \S~9.7 in \cite{BGV} applies almost directly to our context.
\begin{lemma}
The one form component of the differential forms $\alpha^{+}_{\ell}(s)$ 
satisfies
\[
       \overline{\alpha^{+}_{\ell}(s)_{[1]}}= \alpha^{-}_{\ell}(s)_{[1]} 
\]
and has an asymptotic expansion of the form 
\[
    \alpha^{+}_{\ell}(s)_{[1]}\sim \sum^{\infty}_{-N} s^{\frac{k}{2}}(a_{k}+ b_{k} \log s)
\]
as $s\to 0^{+}$.
\label{qc.11}\end{lemma}
\begin{proof}
The asymptotic expansion as $s\to 0^{+}$ follows from the construction of the heat kernel
by Vaillant \cite{Vaillant}, its generalization in \cite{Albin-Rochon} and an application
of the pushforward theorem for manifolds with corners.  
From \eqref{qc.9}, we have that
\begin{equation}
\begin{split}
  \alpha^{+}_{\ell}(s)_{[1]}&= -\frac{1}{2} {}^{R}\Tr_{\pi_{*}\cE^{+}_{\ell}} 
        \lrpar{ D_{\ell} \int_{0}^{1} e^{-(1-\sigma)sD^{2}_{\ell}}[\nabla^{\pi_{*}\cE_{\ell}},D_{\ell}]
                                             e^{-\sigma sD^{2}_{\ell}}d\sigma}  \\
       &=  -\frac{1}{2} {}^{R}\STr_{\pi_{*}\cE_{\ell}} 
        \lrpar{ D_{\ell}^{-} \int_{0}^{1} e^{-(1-\sigma)sD^{2}_{\ell}}[\nabla^{\pi_{*}\cE_{l}},D_{\ell}^{+}]
                                             e^{-\sigma sD^{2}_{\ell}}d\sigma}.                             
\end{split}
\label{qc.12}\end{equation}
Taking the complex conjugate and using the fact that $(D^{+}_{\ell})^{*}= D^{-}_{\ell}$ and that
$\nabla^{\pi_{*}\cE_{\ell}}$ is a unitary connection, we have that
\begin{equation}
\begin{aligned}
\overline{\alpha^{+}_{\ell}(s)_{[1]}} &=  
                              -\frac{1}{2} {}^{R}\Tr_{\pi_{*}\cE^{+}_{\ell}} 
                              \lrpar{ \int_{0}^{1} e^{-\sigma sD^{2}_{\ell}} [D_{\ell}^{-},\nabla^{\pi_{*}\cE_{\ell}}] e^{-(1-\sigma)sD^{2}_{\ell}} D^{+}_{\ell} d\sigma } \\
               &=                \frac{1}{2} {}^{R}\STr_{\pi_{*}\cE_{\ell}} 
                              \lrpar{ \int_{0}^{1} e^{-(1-\sigma)s D^{2}_{\ell}} [D_{\ell}^{-},\nabla^{\pi_{*}\cE_{\ell}}] e^{-\sigma sD^{2}_{\ell}} D^{+}_{\ell} d\sigma } \\
               &=     \frac{1}{2} {}^{R}\STr_{\pi_{*}\cE_{\ell}} 
                              \lrpar{ D^{+}_{\ell} \int_{0}^{1} e^{-(1-\sigma)s D^{2}_{\ell}} [D_{\ell}^{-},\nabla^{\pi_{*}\cE_{\ell}}] e^{-\sigma sD^{2}_{\ell}} d\sigma } \\
               & \quad + \frac{1}{2} {}^{R}\STr_{\pi_{*}\cE_{\ell}} 
                              \lrpar{ \left[\int_{0}^{1} e^{-(1-\sigma)s D^{2}_{\ell}} [D_{\ell}^{-},\nabla^{\pi_{*}\cE_{\ell}}] e^{-\sigma sD^{2}_{\ell}}d\sigma, D^{+}_{\ell} \right] } \\                        &=    - \frac{1}{2} {}^{R}\Tr_{\pi_{*}\cE_{\ell}^{-}} 
                              \lrpar{ D^{+}_{\ell} \int_{0}^{1} e^{-(1-\sigma)s D^{2}_{\ell}} [D_{\ell}^{-},\nabla^{\pi_{*}\cE_{\ell}}] e^{-\sigma sD^{2}_{\ell}} d\sigma } +0 \\
               &= \alpha^{-}_{\ell}(s)_{[1]},               
\end{aligned}
\label{qc.13}\end{equation}
where Lemma~\ref{qc.10} was used in the line before the last one.
\end{proof}
We would like to consider the one-forms
\begin{equation}
  \beta^{\pm}_{\ell}(z) := 2\int_{0}^{\infty} t^{z}\alpha^{\pm}_{\ell}(t)_{[1]} dt.
\label{qc.14}\end{equation}
By Lemma~\ref{qc.11}, this  integral is holomorphic for $\Re z>>0$ and admits a meromorphic extension
to the whole complex plane.  Thus, in this sense, 
$\beta^{+}_{\ell}(z)$ and $\beta^{-}_{\ell}(z)$ are well-defined meromorphic families of one-forms.  We
are interested in their finite part at $z=0$.  More precisely, we would like to consider the one-forms
\begin{equation}
    \beta^{\pm}_{\ell}:= \left.\frac{d}{dz} \frac{1}{\Gamma(z)}\beta^{\pm}_{\ell}(z)\right|_{z=0}
\label{qc.16}\end{equation}
where the evaluation at zero means that we take the finite part at $z=0$.
More generally, we will use the notation
\[
   \fpint \gamma dt:= \lrpar{\frac{d}{dz}\frac{1}{\Gamma(z)}\int_{0}^{\infty} t^{z} \gamma(t) dt}_{z=0}
\]
whenever the integral $\int_{0}^{\infty} t^{z-1} \gamma(t) dt$ varies meromorphically in $z$.
Thus, in this notation,
\[
   \beta_{l}^{\pm}= 2\fpint \alpha^{\pm}_{\ell}(t)_{[1]}dt.
\]
\begin{lemma}
Seen as a function on $T_{g,n}$, the differential of $\zeta'(0;D^{-}_{\ell}D^{+}_{\ell})$ is given by
$d\zeta'(0;D^{-}_{\ell}D^{+}_{\ell})= -\lrpar{ \beta^{+}_{\ell}+\beta^{-}_{\ell}}$.
\label{qc.17}\end{lemma}
\begin{proof}
Using Duhamel's formula, we have that $d\zeta'(0;D^{-}_{\ell}D^{+}_{\ell})$ is given by
\begin{multline}
 \fpint {}^{R}\Tr_{\pi_{*}\cE^{+}_{\ell}}
  \lrpar{ -\frac{1}{t}\int_{0}^{t} e^{-(t-s)D^{-}_{\ell}D^{+}_{\ell}} [\nabla^{\pi_{*}\cE_{\ell}}, 
  D^{-}_{\ell}D^{+}_{\ell}] e^{-sD^{-}_{\ell}D^{+}_{\ell}} ds } dt =\\
  \fpint {}^{R}\Tr_{\pi_{*}\cE^{+}_{\ell}}
  \lrpar{ -\int_{0}^{1} e^{-(1-s)tD^{-}_{\ell}D^{+}_{\ell}} [\nabla^{\pi_{*}\cE_{\ell}}, 
  D^{-}_{\ell}D^{+}_{\ell}] e^{-stD^{-}_{\ell}D^{+}_{\ell}} ds } dt.
\label{qc.18}\end{multline}
On the other hand, we have 
\begin{equation}
\begin{split}
\beta_{\ell}^{+} &= 2 \fpint \alpha^{+}_{\ell}(t)_{[1]} dt  \\
                 &= -\fpint {}^{R}\Tr_{\pi_{*}\cE^{+}_{\ell}}\lrpar{
                 D^{-}_{\ell} \int_{0}^{1} e^{-(1-s)tD^{2}_{\ell}} 
                 [\nabla^{\pi_{*}\cE_{\ell}},D^{+}_{\ell}]e^{-st D^{2}_{\ell}}ds} dt \\
                 &= -\fpint {}^{R}\Tr_{\pi_{*}\cE^{+}_{\ell}}\lrpar{
                  \int_{0}^{1} e^{-(1-s)tD^{2}_{\ell}} 
                  D^{-}_{\ell}[\nabla^{\pi_{*}\cE_{\ell}},D^{+}_{\ell}]e^{-st D^{2}_{\ell}}ds} dt,
\end{split}                 
\label{qc.19}\end{equation}
while
\begin{equation}
\begin{split}
\beta^{-}_{\ell} &= -\fpint {}^{R}\Tr_{\pi_{*}\cE^{-}_{\ell}}\lrpar{
                 D^{+}_{\ell} \int_{0}^{1} e^{-(1-s)tD^{2}_{\ell}} 
                 [\nabla^{\pi_{*}\cE_{\ell}},D^{-}_{\ell}]e^{-st D^{2}_{\ell}}ds} dt \\
                 &= \fpint {}^{R}\STr_{\pi_{*}\cE_{\ell}}\lrpar{
                 D^{+}_{\ell} \int_{0}^{1} e^{-(1-s)tD^{2}_{\ell}} 
                 [\nabla^{\pi_{*}\cE_{\ell}},D^{-}_{\ell}]e^{-st D^{2}_{\ell}}ds} dt \\
                 &= \fpint {}^{R}\STr_{\pi_{*}\cE_{\ell}}\lrpar{
                 \int_{0}^{1} e^{-(1-s)tD^{2}_{\ell}} 
                 [\nabla^{\pi_{*}\cE_{\ell}},D^{-}_{\ell}] D^{+}_{\ell}e^{-st D^{2}_{\ell}}ds} dt \\
                 &\quad + \fpint {}^{R}\STr_{\pi_{*}\cE_{\ell}}\lrpar{
                 \left[D^{+}_{\ell}, \int_{0}^{1} e^{-(1-s)tD^{2}_{\ell}} 
                 [\nabla^{\pi_{*}\cE_{\ell}},D^{-}_{\ell}]e^{-st D^{2}_{\ell}}ds\right]} dt \\
                 &=\fpint {}^{R}\STr_{\pi_{*}\cE_{\ell}}\lrpar{
                 \int_{0}^{1} e^{-(1-s)tD^{2}_{\ell}} 
                 [\nabla^{\pi_{*}\cE_{\ell}},D^{-}_{\ell}] D^{+}_{\ell}e^{-st D^{2}_{\ell}}ds} dt +0,
\end{split}
\label{qc.20}\end{equation}
using Lemma~\ref{qc.10} in the last step.  The result then follows by combining \eqref{qc.18},
\eqref{qc.19} and \eqref{qc.20} and using the formula
\[
  [\nabla^{\pi_{*}\cE_{\ell}}, D^{-}_{\ell}D^{+}_{\ell}]=
   [\nabla^{\pi_{*}\cE_{\ell}}, D^{-}_{\ell}]D^{+}_{\ell} - D^{-}_{\ell}[\nabla^{\pi_{*}\cE_{\ell}}, D^{+}_{\ell}].  
\]
\end{proof}
If $P_{0}: \pi_{*}\cE^{+}_{\ell}\to \ker \bpa_{\ell}$ denotes the orthogonal projection onto the kernel of 
$\bpa_{\ell}$, then the connection 
\begin{equation}
  \nabla^{\ker{\bpa_{\ell}}}= P_{0}\nabla^{\pi_{*}\cE^{+}_{\ell}}P_{0}
\label{qc.21}\end{equation}
is compatible with the $L^{2}$-metric.  It is holomorphic, so that $\nabla^{\ker\bpa_{\ell}}$ is the 
Chern connection of $\ker \bpa_{\ell}$ with respect to the $L^{2}$-metric $\|\cdot\|$.  It defines
a connection on $\det\bpa_{\ell}$, $\nabla^{\det\bpa_{\ell}}$, which is the 
Chern connection of $\det\bpa_{\ell}$ with respect to the $L^{2}$-metric.  We define the 
\textbf{Quillen connection} on $\det \bpa_{\ell}$ to be the connection given by
\begin{equation}
\nabla^{Q_{\ell}}:= \nabla^{\det\bpa_{\ell}}+ \beta^{+}_{\ell}.
\label{qc.22}\end{equation}  
\begin{proposition}
The Quillen connection is the Chern connection of $\det\bpa_{\ell}$ with respect to the Quillen
metric $\|\cdot\|_{Q_{\ell}}$.  
\label{qc.23}\end{proposition}
\begin{proof}
We need to check that $\nabla^{Q_{\ell}}$ is holomorphic and is compatible with the Quillen metric.  To see
that it is holomorphic, it suffices to check that $\beta^{+}_{\ell}$ is a $(1,0)$-form.  Since
$D^{+}_{\ell}= \sqrt{2}\bpa_{\ell}$ is a family of operators that varies 
holomorphically on $T_{g,n}$, the form
\[
               [\nabla^{\pi_{*}\cE_{\ell}},D^{+}_{\ell}]
\]
has to be a $(1,0)$-form (\cf with \eqref{con.18}).  Directly from the definition of $\beta^{+}_{\ell}$, we thus see it
has to be a $(1,0)$-form.

To see that $\nabla^{Q_{\ell}}$ is compatible with the Quillen metric, notice that in general, a 
connection which is compatible with the Quillen metric is of the form
\begin{equation}
   \nabla^{\det\bpa_{l}}-\frac{1}{2}d\zeta'\lrpar{0;D^{-}_{\ell}D^{+}_{\ell}} + \omega
\label{qc.24}\end{equation}
where $\omega$ is any imaginary one-form.  The result then follows by noticing that, taking $\omega=\frac{\beta^{+}_{\ell}-\overline{\beta^{+}_{\ell}}}{2}$ and using lemma~\ref{qc.11}
and lemma~\ref{qc.17}, we get the
Quillen connection.
\end{proof}
We can now compute the curvature of the Quillen connection.

\begin{theorem}
The curvature of the Quillen connection is given by
\begin{equation*}
\begin{split}
\frac{\sqrt{-1}}{2\pi}(\nabla^{Q_{\ell}})^{2}&=  \lrpar{
                   \int_{\cT_{g,n}/T_{g,n}} \Ch\lrpar{ T^{-\ell}(\cT_{g,n}/T_{g,n})} \cdot
                          \Td\lrpar{T\lrpar{\cT_{g,n}/T_{g,n}}} }_{[2]}  \\ 
                & -\sum_{i=1}^{n} \frac{e_{i}}{12}
\end{split}
\end{equation*}
\label{qc.25}\end{theorem}
\begin{proof}
With respect to the connection $\nabla^{\det\bpa_{\ell}}$, we have
\begin{equation}
    \frac{\sqrt{-1}}{2\pi} \lrpar{\nabla^{\det\bpa_{l}}}^{2}= \Ch\lrpar{\nabla^{\ker\bpa_{l}}}_{[2]}.
\label{qc.26}\end{equation}
But by definition, since $\nabla^{\pi_{*}\cE_{\ell}}= \bbA_{[1]}$ for $\bbA$ the Bismut superconnection,
(\cf Proposition 10.16 in \cite{BGV}), we have by \eqref{qc.21} that $\nabla^{\ker\bpa_{l}}$ is the 
connection used in Theorem~\ref{lf.8}.  Thus, $\Ch\lrpar{\nabla^{\ker\bpa_{\ell}}}_{[2]}$ is
given by formula~\eqref{lf.10b}, so that
\begin{equation}
\begin{split}
\frac{\sqrt{-1}}{2\pi}(\nabla^{\det\bpa_{\ell}})^{2}&=  \lrpar{
                   \int_{\cT_{g,n}/T_{g,n}} \Ch\lrpar{ T^{-\ell}(\cT_{g,n}/T_{g,n})} \cdot
                          \Td\lrpar{T\lrpar{\cT_{g,n}/T_{g,n}}} }_{[2]}  \\ 
                & -\sum_{i=1}^{n}\frac{e_{i}}{12}
              -\frac{1}{2\pi\sqrt{-1}} d\int_{0}^{\infty} {}^{R}\STr\lrpar{\frac{\pa \bbA_{t}}{\pa t} e^{-\bbA^{2}_{t}}}_{[1]} dt.
\end{split}
\label{qc.27}\end{equation}
On the other hand, from the definition of the Quillen connection, we have
\begin{equation}
\lrpar{\nabla^{Q_{\ell}}}^{2}=\lrpar{\nabla^{\det\bpa_{l}}}^{2} + d\beta^{+}_{\ell}.
\label{qc.28}\end{equation}
From lemma~\ref{qc.17}, $d(\beta^{+}_{\ell}+ \beta^{-}_{\ell})=0$, hence
\begin{equation}
\lrpar{\nabla^{Q_{\ell}}}^{2}=\lrpar{\nabla^{\det\bpa_{l}}}^{2} + \frac{1}{2}
d\lrpar{\beta^{+}_{\ell}-\beta^{-}_{\ell}}.
\label{qc.29}\end{equation}
But using the fact the Bismut superconnection $\bbA$ is given by $D_{\ell}+ \nabla^{\pi_{*}\cE}$ up
to terms of degree $2$, we have 
\begin{equation}
\begin{split}
\frac{1}{2}\lrpar{\beta^{+}_{\ell}- \beta^{-}_{\ell}} &=
                            \fpint \lrpar{ \alpha_{\ell}^{+}(t)_{[1]} - \alpha_{\ell}^{-}(t)_{[1]}}dt
         = \fpint {}^{R}\STr\lrpar{\frac{\pa \bbA_{t}}{\pa t} e^{-\bbA^{2}_{t}}}_{[1]} dt  \\
         &= \int_{0}^{\infty}    {}^{R}\STr\lrpar{\frac{\pa \bbA_{t}}{\pa t} e^{-\bbA^{2}_{t}}}_{[1]} dt. 
\end{split}
\label{qc.30}\end{equation}
In the last step, we have used the fact ${}^{R}\STr\lrpar{\frac{\pa \bbA_{t}}{\pa t} e^{-\bbA^{2}_{t}}}$ is integrable
in $t$, so that there is no need to regularize.  Combining \eqref{qc.27}, \eqref{qc.29} and 
\eqref{qc.30}, the result follows.
\end{proof}

We should compare our result with the local index formula of 
Takhtajan and Zograf \cite{TZ} 
\begin{equation}
  (\nabla^{Q_{\ell}})^{2}= \frac{6\ell^{2}-6\ell +1}{12\pi^{2}} \omega_{WP}
-\frac{1}{9}\omega_{\TZ},
\label{qc.44}\end{equation}
where $\omega_{WP}$ is the Weil-Peterson K\"ahler form on $T_{g,n}$ and 
$\omega_{\TZ}$ is the K\"ahler form on $T_{g,n}$ defined by Takhtajan and
Zograf in terms of the 
cusp ends of the fibres of $p:\cT_{g,n}\to T_{g,n}$.  A well-known result
of Wolpert \cite{Wolpert}  
(see also p. 424 in \cite{TZ}) shows that 
\begin{multline}
 \lrpar{
                   \int_{\cT_{g,n}/T_{g,n}} \Ch\lrpar{ T^{-\ell}(\cT_{g,n}/T_{g,n})} \cdot
                          \Td\lrpar{T\lrpar{\cT_{g,n}/T_{g,n}}} }_{[2]}=  \\
\frac{6\ell^{2}-6\ell +1}{12\pi^{2}} \omega_{WP}.
\label{qc.45}\end{multline}
Thus, comparing Theorem~\ref{qc.25} with \eqref{qc.44}, we get the following relation. 
\begin{corollary}[Weng \cite{Weng}, Wolpert \cite{Wolpert2}]  For $\ell\ge 0$ and $n>0$, we have
\[
   \widehat{\eta}(D^{V}_{\ell})_{[2]}= 
\sum_{i=1}^{n}\frac{e_{i}}{12}=\frac{1}{9}
         \omega_{\TZ}.
\]
\label{qc.46}\end{corollary}
The fact the Takhtajan-Zograf K\"ahler form is a rational multiple of the 
curvature of a Hermitian line bundle was first obtained by 
Weng \cite{Weng} using Arakelov theory.  This was later improved and 
finalized by Wolpert \cite{Wolpert2}, who obtained more generally that 
 $e_{i}= \frac{4}{3} \omega_{\TZ,i}$ ($\omega_{\TZ,i}$ is defined
in \eqref{WTZ} below) via a natural intrinsic way to
define metrics on the line bundles $\cL_{i}$.

For completeness, let us recall how the Takhtajan-Zograf K\"ahler form
$\omega_{TZ}$ is defined.  Given a fibre $\Sigma$ of 
$p: \cT_{g,n}\to T_{g,n}$, identify it with a quotient of the 
upper half-plane, $\Sigma\cong \Gamma\setminus\bbH$ where $\Gamma$ is the corresponding
Fuchsian group of type $(g,n)$.  Let $\Gamma_{1},\ldots,\Gamma_{n}$ be
the list of non-conjugate parabolic subgroup of $\Gamma$ as in \eqref{con.2}
so that 
\[
                   \sigma_{i}^{-1}\Gamma_{i}\sigma_{i}=\Gamma_{\infty}
\]
for $i\in\{1,\ldots,n\}$.  The Eisenstein-Mass series 
$E_{i}(z,s)$ associated to the $i^{\text{th}}$ cusp of the group $\Gamma$ is defined
for $\Re s>1$ by the formula
\begin{equation}
  E_{i}(z,s):= \sum_{\gamma\in \Gamma_{i}\setminus \Gamma} 
\Im( \sigma^{-1}_{i}\gamma z)^{s}.
\label{qc.47}\end{equation}
The Eisenstein-Mass series naturally descends to define a function
on the quotient $\Sigma=\Gamma\setminus \bbH$. 
Recall that under the identification of $T_{[\Sigma]}T_{g,n}$ with the
space of harmonic Beltrami differentials $\Omega^{-1,1}(\Sigma)$, the 
Weil-Peterson K\"ahler metric is defined by
\begin{equation}
  \langle\mu,\nu\rangle_{WP}:= \int_{\Sigma} \mu \overline{\nu}dg_{\Sigma}
=\int_{\Sigma} \langle \mu,\nu\rangle_{K^{-1}\otimes \Lambda^{0,1}_{\Sigma}}
dg_{\Sigma} 
\label{qc.48}\end{equation}   
for $\mu,\nu\in T_{[\Sigma]}T_{g,n}$ with corresponding K\"ahler form
given by
\begin{equation}
  \omega_{WP}(\mu,\overline{\nu})= \frac{\sqrt{-1}}{2} 
   \langle \mu, \nu\rangle_{WP}.
\label{qc.49}\end{equation} 
To define their K\"ahler metric, Takhtajan and Zograf considered instead
\begin{equation}
\langle \mu, \nu\rangle_{i}=\int_{\Sigma} \mu \overline{\nu}
  E_{i}(\cdot,2)dg_{\Sigma}, \quad i=1,\ldots n.
\label{qc.50}\end{equation}
Each of these scalar products gives rise to a K\"ahler metric on $T_{g,n}$
with corresponding K\"ahler form 
\begin{equation}
  \omega_{\TZ,i}(\mu,\overline{\nu})= \frac{\sqrt{-1}}{2} 
   \langle \mu, \nu\rangle_{i}, \quad i=1,\ldots,n.
\label{WTZ}\end{equation} 
The sum of these metric is the Takhtajan-Zograf K\"ahler metric
\begin{equation}
         \langle\mu,\nu\rangle_{\TZ}:= \sum_{i=1}^{n}
\langle \mu, \nu\rangle_{i}
\label{qc.51}\end{equation}
with corresponding K\"ahler form given by 
\begin{equation}
\omega_{\TZ}(\mu,\overline{\nu})= \frac{\sqrt{-1}}{2} 
   \langle \mu, \nu\rangle_{\TZ}.
\label{qc.52}\end{equation}

We know from Corollary~\ref{qc.46} that the eta form
$\widehat{\eta}(D^{V}_{\ell})_{[2]}$ is the K\"ahler form of a 
K\"ahler metric.  This is consistent with Theorem~\ref{lf.8} asserting
that the eta form $\widehat{\eta}(D_{\ell}^{V})$ is closed.


\providecommand{\bysame}{\leavevmode\hbox to3em{\hrulefill}\thinspace}
\providecommand{\MR}{\relax\ifhmode\unskip\space\fi MR }
\providecommand{\MRhref}[2]{%
  \href{http://www.ams.org/mathscinet-getitem?mr=#1}{#2}
}
\providecommand{\href}[2]{#2}

\end{document}